\documentclass[a4paper,10pt]{article}
\usepackage{mathtext}
\usepackage[T1,T2A]{fontenc}
\usepackage[cp1251]{inputenc}
\usepackage[english]{babel}
\usepackage{amsmath}
\usepackage{amsfonts}
\usepackage{amssymb}
\usepackage{mathrsfs}
\usepackage{amsthm}
\usepackage{enumerate}
\usepackage{graphicx}

\usepackage{color}
\usepackage{euscript}

\usepackage{cite}

\newtheorem{Le}{Lemma}[section]
\newtheorem{Def}[Le]{Definition}
\newtheorem{St}[Le]{Proposition}
\newtheorem{Th}{Theorem}[section]
\newtheorem{Cor}[Le]{Corollary}
\newtheorem{Rem}[Le]{Remark}

\newtheorem{Ex}[Le]{Example}
\numberwithin{equation}{section}

\newcommand{\R}{\mathbb{R}}
\newcommand{\Co}{\mathbb{C}}
\newcommand{\N}{\mathbb{N}}
\newcommand{\Z}{\mathbb{Z}}

\newcommand{\eps}{\varepsilon}

\newcommand{\eq}[1]{\begin{equation}{#1}\end{equation}}
\newcommand{\mlt}[1]{\begin{multline}{#1}\end{multline}}
\newcommand{\alg}[1]{\begin{align}{#1}\end{align}}

\newcommand{\set}[2]{\{{#1}\mid{#2}\}}
\newcommand{\Set}[2]{\Big\{{#1}\,\Big|\;{#2}\Big\}}
\newcommand{\scalprod}[2]{\langle{#1},{#2}\rangle}
\newcommand{\fdot}{\,\cdot\,}
\newcommand{\Eeqref}[1]{\stackrel{\scriptscriptstyle{\eqref{#1}}}{=}}
\newcommand{\Leqref}[1]{\stackrel{\scriptscriptstyle{\eqref{#1}}}{\leq}}
\newcommand{\Lseqref}[1]{\stackrel{\scriptscriptstyle{\eqref{#1}}}{\lesssim}}
\newcommand{\LeqrefTwo}[2]{\stackrel{\scriptscriptstyle{\eqref{#1},\eqref{#2}}}{\leq}}

\newcommand{\Lref}[1]{\stackrel{#1}{\leq}}
\newcommand{\Lsref}[1]{\stackrel{#1}{\lesssim}}

\DeclareMathOperator{\BV}{BV}
\DeclareMathOperator{\I}{I}
\DeclareMathOperator{\spec}{spec}
\DeclareMathOperator{\Heat}{H}
\DeclareMathOperator{\loc}{loc}
\DeclareMathOperator{\s}{s}
\DeclareMathOperator{\Lip}{Lip}
\newcommand{\Disp}{\mathbb{D}}
\newcommand{\E}{\mathbb{E}}
\newcommand{\T}{\mathcal{T}}
\DeclareMathOperator{\Hardy}{H}
\newcommand{\mass}{{\bf \mathrm{m}}}

\DeclareMathOperator{\Conv}{Co}
\DeclareMathOperator{\Fl}{Fl}

\newcommand{\Me}{\boldsymbol{\mathrm{M}}}
\newcommand{\MM}{\mathbb{M}}
\newcommand{\Sw}{\mathcal{S}}
\newcommand{\W}{\boldsymbol{\mathrm{W}}}
\newcommand{\WW}{\mathcal{W}}
\newcommand{\LIP}{\mathfrak{L}}
\DeclareMathOperator{\M}{M}

\title{Hardy--Littlewood--Sobolev inequality for $p=1$}
\author{Dmitriy Stolyarov\thanks{Supported by RFBR grant no. 18-31-00037.}}
\begin{document}
\maketitle
\begin{abstract}
Let~$\WW$ be a closed dilation and translation invariant subspace of the space of~$\R^\ell$-valued Schwartz distributions in~$d$ variables. We show that if the space~$\WW$ does not contain distributions of the type~$a\otimes \delta_0$,~$\delta_0$ being the Dirac delta, then the inequality~$\|\I_\alpha [f]\|_{L_{p,1}}\lesssim \|f\|_{L_1}$,~$\frac{p-1}{p} = \frac{\alpha}{d}$, holds true for functions~$f\in\WW\cap L_1$ with a uniform constant; here~$\I_\alpha$ is the Riesz potential of order~$\alpha$ and~$L_{p,1}$ is the Lorentz space. This result implies as a particular case the inequality~$\|\nabla^{m-1} f\|_{L_{\frac{d}{d-1},1}} \lesssim \|A f\|_{L_1}$, where~$A$ is a canceling elliptic differential operator of order~$m$. 
\end{abstract}

\section{Generalized Sobolev and~$\BV$ spaces}\label{S1}
Let~$l$ and~$d$ be natural numbers. We will be working with functions that map~$\R^d$ to~$\Co^l$. We equip the latter space with the standard Euclidean norm on~$\R^{2l}$:
\eq{|a|^2 = \sum\limits_{j=1}^l a_j\bar{a}_j,\qquad a\in \Co^l.
}
Let~$p \in [1,\infty)$. Consider the space~$L_p(\R^d,\Co^l)$ of measurable functions~$f\colon \R^d\to \Co^l$ such that the quantity 
\eq{\label{LpNormDef}
\|f\|_{L_p(\R^d,\Co^l)} = \Big(\int\limits_{\R^d} |f(x)|^p\,dx\Big)^{\frac1p}
}
is finite. 
By the usual limit argument, one may extend this definition to the case~$p=\infty$. We may further introduce the space~$\Me(\R^d,\Co^l)$ that consists of all charges ($\Co^l$-valued sigma-additive Borel set functions) of finite variation. Here and in what follows we distinguish measures that are always scalar and non-negative from charges, which may be either~$\R$,  or~$\Co$, or~$\R^\ell$, or~$\Co^l$ valued. We define the norm in the~$\Me$ space by the formula
\eq{\label{TotalVariation}
\|\mu\|_{\Me(\R^d,\Co^l)} = \sup \Set{\int\limits_{\R^d} f\,d\mu}{\|f\|_{C_0(\R^d,\Co^l)} \leq 1}.
}
Here~$C_0(\R^d,\Co^l)$ is the space of continuous functions that tend to zero at infinity, equipped with the standard sup-norm. Note that the norm~\eqref{TotalVariation} coincides with the total variation of~$\mu$ defined in the usual way.

Let~$k \leq l$ be a natural number. Let~$\Omega\colon S^{d-1}\to G(l,k)$ be a smooth mapping. The notation~$S^{d-1}$ and~$G(l,k)$ means the unit sphere in~$\R^d$ and the (complex) Grassmannian, i.e. the set of all (complex) linear~$k$-dimensional subspaces of~$\Co^l$. The map~$\Omega$ gives rise to a generalization of the Sobolev space
\eq{
W_1^\Omega = \Set{f\in L_1(\R^d,\Co^l)}{\forall \xi \in \R^d \setminus \{0\} \quad \hat{f}(\xi) \in \Omega\big(\frac{\xi}{|\xi|}\big)}
}
and the~$\BV$-space
\eq{
\BV^\Omega = \Set{\mu\in \Me(\R^d,\Co^l)}{\forall \xi \in \R^d \setminus \{0\} \quad \hat{\mu}(\xi) \in \Omega\big(\frac{\xi}{|\xi|}\big)}.
}
These spaces inherit the norms from the spaces~$L_1$ and~$\Me$ correspondingly. Here and in what follows we use the standard Harmonic Analysis normalization of the Fourier transform
\eq{
\hat{f}(\xi) =\mathcal{F}[f](\xi) = \int\limits_{\R^d} f(x)e^{-2\pi i\scalprod{\xi}{x}}\,dx;\quad \hat{\mu}(\xi) = \int\limits_{\R^d}e^{-2\pi i \scalprod{\xi}{x}}d\mu(x).
} 
Since we are working with the Fourier transform, we will need the Schwartz class. We denote it by~$\Sw(\mathbb{R}^d)$ or~$\Sw(\R^d,\Co^l)$ depending on whether we consider scalar or vector valued functions.

\begin{Rem}
The spaces~$W_1^\Omega$ and~$\BV^\Omega$ are closed in~$L_1(\R^d,\Co^l)$ and~$\Me(\R^d,\Co^l)$ respectively. These spaces are also translation and dilation invariant. 
\end{Rem}
\begin{Ex}\label{GagliardoNirenbergExample}
Let~$l=d$ and~$k=1$. Consider the mapping
\eq{
\Omega(\zeta) = \Co\zeta,\quad \zeta \in S^{d-1},
}
i.e. the vector~$\zeta$ is mapped to the complex line spanned by~$\zeta$. In this case\textup,
\eq{
W_1^{\Omega} = \{\nabla f\mid f\in \dot{W}_{1}^1(\R^d)\};\quad \BV^\Omega = \{\nabla f\mid f\in \BV(\R^d)\}.
}
In other words\textup, the classical spaces~$\dot{W}_1^1$ and~$\BV$ may be obtained by choosing specific~$\Omega$.
\end{Ex}
\begin{Ex}\label{DivFree}
Let~$l=d$ and~$k=d-1$. Define the mapping~$\Omega$ by the formula
\eq{
\Omega(\zeta) = \Set{\eta \in \R^d}{\sum\limits_{j=1}^d \zeta_j\eta_j = 0},
}
i.e.~$\Omega(\zeta)$ is the orthogonal complement of the line spanned by~$\zeta$. In this case\textup,~$\BV^\Omega$ is the space of divergence free \textup(solenoidal\textup) charges.
\end{Ex}
\begin{Ex}\label{JeanExample}
One may go further and consider a vectorial homogeneous of order~$m$ elliptic differential operator~$A$ that maps~$V$-valued functions to~$E$-valued functions\textup, here~$V$ and~$E$ are finite dimensional spaces. Let~$\mathcal{L}(V,E)$ be the space of all linear operators with domain~$V$ and image in~$E$. One may think of~$A$ in terms of its symbol~$\mathbb{A}$ that is a mapping~$\mathbb{A}\colon \mathbb{R}^d \to \mathcal{L}(V,E)$ such that
\eq{
A[f] = \mathcal{F}^{-1}\Big[\mathbb{A} (\xi)[\hat{f}(\xi)]\Big],\quad f\in \Sw(\mathbb{R}^d,V).
}
Since we assume~$A$ is homogeneous\textup, the mapping~$\mathbb{A}$ is a homogeneous \textup(matrix-valued\textup) polynomial of order~$m$. The associated function~$\Omega$ is defined by the formula
\eq{\label{OmegaAndDiffOperators}
\Omega(\zeta) = \mathrm{Im}\,\mathbb{A}(\zeta),\quad \zeta \in S^{d-1}. 
}
Since~$A$ is elliptic\textup, the image of~$V$ has dimension~$\dim V$ for any~$\zeta \in S^{d-1}$\textup, and we indeed get a smooth mapping into~$G(\dim E,\dim V)$. The corresponding spaces~$W_1^{\Omega}$ and~$\BV^\Omega$ are usually denoted by~$W_1^{A}$ and~$\BV^{A}$.  

In the case~$A = \nabla$ considered in Example~\textup{\ref{GagliardoNirenbergExample}}\textup,~$V = \mathbb{C}$\textup,~$E = \mathbb{C}^d$\textup, and
\eq{
\mathbb{A}(\zeta)[\lambda] = 2\pi i \zeta \lambda,\quad \zeta \in \mathbb{R}^d,\ \lambda \in V = \mathbb{C}.
}
The case considered in Example~\textup{\ref{DivFree}} corresponds to the differential operator~$A = \mathrm{curl}$\textup, here~$V = E =\mathbb{C}^d$\textup; note that this operator is not elliptic\textup, it is a constant rank operator only.
\end{Ex}

Let~$\alpha \in (0,d)$. Consider the Riesz potential~$\I_\alpha$,
\eq{
\I_{\alpha}[\mu] = \mathcal{F}^{-1}\Big[|\cdot|^{-\alpha}\hat{\mu}\Big],\quad \mu \in \Me(\R^d).
}
We may define the action of~$\I_\alpha$ on vector-valued functions and charges simply applying it to each coordinate individually. 

A simple computation shows that~$\I_{\alpha}[\delta_0] \notin L_{\frac{d}{d-\alpha}}$, where~$\delta_x$ is the Dirac delta at~$x\in\R^d$, and as a consequence,~$\I_{\alpha}\colon L_1 \nrightarrow L_{\frac{d}{d-\alpha}}$. In other words, the Hardy--Littlewood--Sobolev inequality fails at~$p=1$ (in the present text~$p$ stands for the parameter on the left hand side of the inequality, which is usually denoted by~$q$, since the summability parameter on the right hand side is always equal to one in our considerations). See Chapter $5$ in~\cite{Stein1970} for the original Hardy--Littlewood--Sobolev inequality and its applications.

\begin{Def}
We say that~$\Omega$ satisfies the cancellation condition if
\eq{\label{Cancel}
\bigcap_{\zeta \in S^{d-1}}\Omega(\zeta) = \{0\}.
}
\end{Def}
This condition was introduced by Roginskaya--Wojciechowski in~\cite{RoginskayaWojciechowski2006} and Van Schaftingen in~\cite{VanSchaftingen2013} independently. 

The space~$\Me(\R^d,\Co^l)$ has natural tensor product structure. If~$a \in \Co^l$ and~$\mu\in \Me(\R^d)$ is a scalar-valued charge, then the charge~$a\otimes \mu$ is defined by the formula
\eq{
a\otimes \mu (B) = \mu(B)a,
}
where~$B\subset \R^d$ is an arbitrary Borel set. 
\begin{Rem}
The cancellation condition~\eqref{Cancel} is equivalent to the absence of the charges~$a\otimes \delta_0$\textup,~$a\in \Co^l \setminus \{0\}$\textup, in the space~$\BV^\Omega$.
\end{Rem}
\begin{Th}\label{OurSobolevEmbedding}
Let~$\Omega$ satisfy the cancellation condition~\eqref{Cancel}. Then\textup,~$\I_{\alpha}\colon W_1^\Omega \to \dot{B}_{\frac{d}{d-\alpha},1}^{0,1}$ when~$\alpha \in (0,d)$.
\end{Th}
The space~$B$ in the above theorem is the Besov--Lorentz space with the best parameters possible for such an embedding. For details on Besov--Lorentz spaces, see~\cite{Peetre1976}, we provide a brief outline only. The norm in this space is defined by the rule
\eq{\label{BesovLorentzFormula}
\|g\|_{\dot{B}_{p,1}^{0,1}} = \sum\limits_{k\in\Z} \|g*(\psi_k - \psi_{k-1})\|_{L_{p,1}},
}
where the functions~$\psi_k(x) = A^{dk}\psi(A^{k}x)$ form an approximate identity constructed from a smooth function~$\psi$ whose Fourier transform equals one in a neighborhood of the origin and is compactly supported; here~$A > 1$ is an auxiliary parameter (different choices of~$A$ lead to equivalent norms). The definition of the Lorentz semi-norm we use may be found in~\eqref{LorentzDef} below. By the limit relations
\eq{
g*\psi_k \stackrel{\scriptscriptstyle L_{p,1}}{\longrightarrow} g, k \to \infty, \quad g*\psi_k \stackrel{\scriptscriptstyle L_{p,1}}{\longrightarrow} 0, k\to -\infty,
}
and the triangle inequality in~$L_{p,1}$ (note that~$p > 1$),
\eq{
\|g\|_{L_{p,1}} \lesssim \|g\|_{\dot{B}_{p,1}^{0,1}}, 
}
so Theorem~\ref{OurSobolevEmbedding} leads to the following corollary.
\begin{Cor}\label{LorentzCorollary}
Let~$\Omega$ satisfy the cancellation condition~\eqref{Cancel}. Then\textup,~$\I_{\alpha}\colon W_1^\Omega \to L_{\frac{d}{d-\alpha},1}$ when~$\alpha \in (0,d)$.
\end{Cor}
We have used the notation~$\lesssim$. We use it in the following meaning:~$A \lesssim B$ signifies there exists a constant~$C$ that does not depend on certain parameters and such that~$A \leq C B$. The said independence is usually either discussed somewhere nearby or is clear from the context.

In the case of rational~$\Omega$ considered in Example~\ref{JeanExample} and~$\alpha = 1$, Corollary~\ref{LorentzCorollary} solves Open Problem~$8.3$ in~\cite{VanSchaftingen2013}. The particular divergence free case (as in Example~\ref{DivFree}) has been recently considered in~\cite{HernandezSpector2020}, solving Open Problem~$1$ in~\cite{BourgainBrezis2007}. Corollary~\ref{LorentzCorollary}, in its turn, leads to a form of Hardy's inequality since~$|\fdot|^{-\alpha}$ belongs to~$L_{\frac{d}{\alpha},\infty}$, which is the dual space to~$L_{\frac{d}{d-\alpha},1}$.
\begin{Cor}\label{HardyCorollary}
If~$\Omega$ satisfies the cancellation condition~\eqref{Cancel} and~$\alpha \in (0,d)$\textup, then
\eq{
\int\limits_{\R^d}\frac{|\I_\alpha [f](x)|}{|x-x_0|^{\alpha}}\,dx \lesssim \|f\|_{W_1^\Omega}
}
for any~$x_0 \in \R^d$ and~$f\in W_1^\Omega$.
\end{Cor}
One may use the embedding~$L_{p,1}\hookrightarrow L_p$ to obtain yet another corollary that solves Open Problem~$8.2$ in~\cite{VanSchaftingen2013} (see also Open Problem~$1$ in~\cite{VanSchaftingen2010}).
\begin{Cor}
Let~$\Omega$ satisfy the cancellation condition~\eqref{Cancel}. Then\textup,~$\I_{\alpha}\colon W_1^\Omega \to \dot{B}_{\frac{d}{d-\alpha}}^{0,1}$ when~$\alpha \in (0,d)$.
\end{Cor}

Since~$\I_{\alpha}[\delta_0] \notin L_{\frac{d}{d-\alpha}}$ for~$\alpha \in (0,d)$, condition~\eqref{Cancel} in Theorem~\ref{OurSobolevEmbedding} is necessary when~$\alpha \in (0,d)$. For the limit case~$\alpha = d$, it is not, see~\cite{Raita2019} and~\cite{Stolyarov2020} for details. Our considerations also lead to interesting information in this case, see Remark~\ref{plessthantwo} below.


Theorem~\ref{OurSobolevEmbedding} (and our main Theorem~\ref{Main} below) has many predecessors. The Hardy--Littlewood--Sobolev inequality was invented by Sobolev in~\cite{Sobolev1938} as a tool to prove what is now called the Sobolev embedding theorem; unfortunately, his method did not work for the limit summability exponent~$1$. The simplest and the most natural case mentioned in Example~\ref{GagliardoNirenbergExample} is equivalent via the Calder\'on--Zygmund theory to a certain limiting Sobolev embedding for the space~$\dot{W}_1^1(\R^d)$. If~$\alpha = 1$ and we embed into Lebesgue space~$L_{\frac{d}{d-1}}$, we get the classical Gagliardo--Nirenberg inequality obtained by Gagliardo~\cite{Gagliardo1959} and Nirenberg~\cite{Nirenberg1959}; see the book~\cite{Mazya2011} for more information about this classical case. The embedding into the best possible Lorentz space~$L_{\frac{d}{d-1},1}$ was first proved by Alvino in~\cite{Alvino1977} and then rediscovered by Poornima~\cite{Poornima1983} and Tartar (see~\cite{Tartar1998} for more historical remarks on this question). For higher order smoothnesses and Besov spaces, the complete result was obtained by Kolyada in~\cite{Kolyada1993}, see the papers~\cite{BesovIlin1969} and~\cite{Solonnikov1972} for earlier results and~\cite{Spector2020} for a different approach. We note that most of these results are formulated in the more general anisotropic setting, i.e. when derivates with respect to different coordinates may have different orders. We refer the reader to the book~\cite{BIN1978} for details on anisotropic theory.

The inequality
\eq{\label{SolenoidSobolev}
\|\I_1[f]\|_{L_{\frac{d}{d-1}}} \lesssim \|f\|_{L_1},\quad \mathrm{div} f=0,
}
was obtained by Bourgain and Brezis in~\cite{BourgainBrezis2004}. In a sense,  this inequality served as a turning point for the whole theory.  The proof in~\cite{BourgainBrezis2004} relied upon Smirnov's theorem from~\cite{Smirnov1994}, which decomposes an arbitrary solenoidal charge into an integral of currents tangent to smooth curves, and a particular case of~\eqref{SolenoidSobolev} for such special charges already proved in~\cite{BBM2004} . Another proof was suggested by Van Schaftingen in~\cite{VanSchaftingen2004}; see the papers~\cite{BourgainBrezis2002} and~\cite{VanSchaftingen2004one} for related results. The inequality
\eq{\label{SolenoidSobolevLorentz}
\|\I_1[f]\|_{L_{\frac{d}{d-1},1}} \lesssim \|f\|_{L_1},\quad \mathrm{div} f=0,
}
was a long-standing conjecture (the question goes back to~\cite{BourgainBrezis2007}) until Hernandez and Spector have resolved it in~\cite{HernandezSpector2020}. Their proof also relies upon Smirnov's theorem. According to the knowledge of the author, the inequality
\eq{
\|\I_1[f]\|_{\dot{B}^{0,1}_{\frac{d}{d-1},1}} \lesssim \|f\|_{L_1},\quad \mathrm{div} f=0,
}
which is a complete form of Theorem~\ref{OurSobolevEmbedding} in this case, is unknown. 

The case of general differential operators as in Example~\ref{JeanExample},~$\alpha = 1$, and the Lebesgue space instead of Besov--Lorentz was considered by Van Schaftingen in~\cite{VanSchaftingen2013} (see earlier papers ~\cite{Strauss1973} (symmetric gradient),~\cite{LanzaniStein2005} (Hodge differentials),~\cite{Mazya2007} (sharp constants in some of these inequalities),~\cite{BourgainBrezis2007} (higher order derivatives and related approximation problems), for particular cases, and the surveys~\cite{VanSchaftingen2014} and~\cite{Spector2019} for more historical details). The differential operators satisfying~\eqref{Cancel} are called canceling; the ellipticity condition might be replaced by a weaker constant rank condition, see~\cite{Raita2018}.
For results on Hardy's inequalities as in Corollary~\ref{HardyCorollary}, see~\cite{Mazya2010} and~\cite{BousquetVanSchaftingen2014}. For results on Lorentz spaces~$L_{\frac{d}{d-1},1}$ in the case of first order operators~$A$, see~\cite{SpectorVanSchaftingen2019}. For generalizations to metric spaces other than~$\R^d$ see~\cite{CVSYu2017}.

There is also a related problem about the sharp estimates of singularities of measures in~$\BV^\Omega$. The question is: 'What is the best possible bound from below for the lower Hausdorff dimension of~$\mu \in \BV^\Omega$?' It was raised in~\cite{RoginskayaWojciechowski2006}.  We note that if~$\Omega$ is purely antisymmetric (that is~$\Omega(\zeta)\cap \Omega(-\zeta) = \{0\}$ for any~$\zeta \in S^{d-1}$), then both~$\BV^\Omega$ and~$W_1^\Omega$ are contained in the real Hardy class~$\mathcal{H}_1(\mathbb{R}^d,\mathbb{C}^l)$ by the celebrated Uchiyama theorem~\cite{Uchiyama1982}; the necessity of the antisymmetry condition was observed by Janson earlier in~\cite{Janson1977}. Therefore, any measure~$\mu \in \BV^\Omega$ is absolutely continuous with respect to the Lebesgue measure. The question for general~$\Omega$ seems to be open. Partial results were obtained in~\cite{APHF2019},~\cite{AyoushWojciechowski2017},~\cite{RoginskayaWojciechowski2006}, and~\cite{StolyarovWojciechowski2014}. The author supposes that the methods of the present text may also help in this related problem, see the preprint~\cite{Stolyarov2020bis}.

For a related problem on trace inequalities, see~\cite{GRvS2019} (the classical case of the first gradient may be found in~\cite{Mazya2011}).  Theorem~\ref{OurSobolevEmbedding} (via trace inequalities for Riesz potentials, see~\cite{AdamsHedberg1999}) implies the following 'trace' theorem.
\begin{Cor}
Let~$\Omega$ satisfy the cancellation condition~\eqref{Cancel} and let~$\alpha \in (0,d)$. Then\textup, the embedding~$\I_{\alpha}\colon W_1^\Omega \to L_q(\mu)$ is continuous provided~$q > 1$ and the measure~$\mu$ satisfies the Frostman-type condition
\eq{
\mu(B_r(x))^{\frac{1}{q}} \lesssim r^{d-\alpha}
}
for any radius~$r>0$ and center~$x\in \R^d$ of the Euclidean ball~$B_r(x)$ \textup(with uniform constants\textup).
\end{Cor}

We have already said that the classical embedding theorems (i.e. for classical Sobolev spaces) allow anisotropic generalizations (see~\cite{Kolyada1993} for the strongest possible result in the classical setting). Some partial results in the anisotropic setting and spaces of functions in the style of~$W_1^\Omega$ were obtained in~\cite{KislyakovMaximov2018},~\cite{KMS2015}, and~\cite{Stolyarov2020}; see the second of these papers for applications to questions in the Banach space theory (the idea that inequalities of the type we discuss might deliver interesting information about the isomorphic type of Banach spaces goes back to~\cite{Kislyakov1975} and~\cite{LindenstraussPelczynski1968}). The author supposes that the methods of the present text may be transferred to the anisotropic setting.

The techniques we will use are different from those usually used in the field; however, there is some similarity to~\cite{BourgainBrezis2002} and~\cite{BourgainBrezis2007}. We will rely upon Harmonic Analysis tools such as the time-frequency decomposition and Harnack's inequality (we will provide a more detailed description of our methods at the end of this section). This allows us to get rid of differential structure and work in a more general setting of Fourier restrictions. 
\begin{St}\label{Density}
The set of Schwartz~$\Co^l$-valued functions is dense in~$W_1^\Omega$ and is $*$-weak dense in~$\BV^\Omega$ \textup(in the $*$-weak topology inherited from~$\Me(\R^d,\Co^l)$\textup).
\end{St}
Formally, this proposition is not needed for the proof of Theorem~\ref{OurSobolevEmbedding}. We present it here since it leads to a useful definition and a more elegant generalization of Theorem~\ref{OurSobolevEmbedding}, that is Theorem~\ref{Main} below.

For any~$L \in G(l,k)$, we denote by~$\pi_L$ the orthogonal projection of~$\Co^l$ onto~$L$.
\begin{proof}[Proof of Proposition~\ref{Density}]
We deal with approximation in~$W_1^\Omega$, the approximation in the space of charges may be obtained with a similar reasoning (in fact, the two cases do not differ at all after the first step, since a charge whose Fourier transform is compactly supported is absolutely continuous with respect to the Lebesgue measure). Let~$\{\Phi_n\}_n$ be an approximate identity constructed from a smooth compactly supported scalar function~$\Phi$ with unit integral in the usual way:
\eq{
\Phi_n(\xi) = n^d \Phi(n\xi),\qquad \xi\in\R^d,\ n\in\N.
}
Let us also require that~$\Phi = 1$ in a neighborhood of the origin. Pick an arbitrary~$f \in W_1^\Omega$. Our target is to construct a sequence of~$W^\Omega_1\cap \Sw(\R^d,\Co^l)$ functions that approximate~$f$.

First, we may assume~$\hat{f}$ is compactly supported since the functions~$(\hat{f}(\xi)\Phi(\xi/n))\check{\phantom{i}}$ belong to~$W^\Omega_1$, have compactly supported Fourier transforms, and approximate~$f$ in~$L_1$ norm.

Second, we may assume~$\int_{\R^d} f = 0$. Indeed, if~$\int f = a \in \Co^l$, then
\eq{
a\in\!\!\! \bigcap_{\xi\in S^{d-1}}\!\!\!\Omega(\xi).
}
Thus, we may replace~$f$ with~$f - a\otimes \Phi_0$ since~$a\otimes \Phi_0 \in W_1^\Omega$ in this case.

Third, we may assume~$0\notin \spec f$. To justify it, we note that the functions~$f_n = f - (\hat{f}(\xi) \Phi(n\xi))\check{\phantom{i}}$ belong to~$W_1^\Omega$ and approximate~$f$ in~$L_1$ norm provided~$\int_{\R^d} f =0$. Since we have required~$\Phi = 1$ in a neighborhood of the origin, we have~$0\notin \spec f_n$.

With all the assumptions at hand, we construct the approximations by the formula
\eq{
f_n = \mathcal{F}^{-1}\Big[\pi_{\Omega(\xi/|\xi|)}\big[\hat{f}*\Phi_n(\xi)\big]\Big].
}
The projection inside the formula guarantees~$f_n \in W_1^\Omega$. We note that the operator of "Fourier projection", that is~$g\mapsto (\pi_{\Omega(\xi/|\xi|)}[\hat{g}(\xi)])\check{\phantom{i}}$, is bounded in the~$L_1\to L_1$ norm as long as the spectrum of the function~$g$ is separated from the origin and the infinity. Therefore, if~$n$ is sufficiently large, then
\eq{
\|f - f_n\|_{L_1} = \Big\|\mathcal{F}^{-1}\Big[\pi_{\Omega(\xi/|\xi|)}\big[\hat{f}(\xi) - \hat{f}*\Phi_n(\xi)\big]\Big]\Big\|_{L_1} \lesssim  \Big\|\mathcal{F}^{-1}\Big[\hat{f} - \hat{f}*\Phi_n\Big]\Big\|_{L_1} \to 0,
}
which is true since~$\check{\Phi}(0)=1$.
\end{proof}
\begin{Cor}
If Theorem~\textup{\ref{OurSobolevEmbedding}} is true\textup, then~$\I_{\alpha}$ also maps~$\BV^\Omega$ to~$\dot{B}_{\frac{d}{d-\alpha},1}^{0,1}$.
\end{Cor}
The proof of Proposition~\ref{Density} hints a good way to define a space of~$\Omega$-subordinate distributions, which plays an important role in the subject. Such an object was introduced by Ayoush and Wojciechowski in~\cite{AyoushWojciechowski2017}.
\begin{Def}
Define the space~$\W$ by the rule
\eq{
\W = \Set{f\in \Sw'(\R^d,\Co^l)}{ \pi_{\Omega(\xi/|\xi|)^{\perp}}[\hat{f}]\cdot H = 0}.
}
\end{Def}
Here~$H$ is an auxiliary scalar Schwartz function that has deep zero at the origin (i.e., for any~$N\in \N$ the relation~$H(x) = O(|x|^{N})$ holds true when~$x\to 0$) and is positive outside it. The function~$H$ is used to make the formula correct, it kills the non-smoothness of the projection at the origin. Note that the space~$\W$ contains all polynomials. 
\begin{Rem}\label{WtildeRemark}
The space~$\W$ is dilation and translation invariant. It is closed as a subspace of~$\Sw'(\R^d,\Co^l)$. Note that the cancellation condition~\eqref{Cancel} is equivalent to the absence of the charges~$a\otimes \delta_0$\textup,~$a\in \Co^l \setminus \{0\}$\textup, in the space~$\W$. 
\end{Rem}

There will be no (or very little) Fourier transform in the forthcoming sections, so we replace~$\Co^l$ with~$\R^{2l}$. Let~$\ell = 2l$. We will never use that~$\ell$ is even. We a ready to formulate our main result.
\begin{Th}\label{Main}
Let~$\WW$ be a closed linear subspace of~$\Sw'(\R^d,\R^\ell)$ that is invariant under translations and dilations\textup, let~$\alpha \in (0,d)$. The constant in the inequality 
\eq{
\|\I_\alpha [f]\|_{\dot{B}_{\frac{d}{d-\alpha},1}^{0,1}}\lesssim \|f\|_{L_1},\qquad f\in\WW,
}
is uniform with respect to all~$f\in \WW$\textup, for which the right hand side is finite\textup, if and only if~$\WW$ does not contain the charges~$a\otimes \delta_0$\textup,~$a\in \R^\ell \setminus \{0\}$.
\end{Th}
From now on let~$\WW$ be a closed linear subspace of~$\Sw'(\R^d,\R^\ell)$ that is invariant under translations and dilations. Seemingly, one may require closedness in a stronger topology (of course, not as strong as the~$*-$weak topology in~$\Me$), which will lead to milder smoothness assumptions on~$\Omega$ in Theorem~\ref{OurSobolevEmbedding}. The Schwartz class is chosen mostly for convenience.
\begin{Rem}\label{plessthantwo}
By the classical Besov embedding
\eq{
\I_{\beta - \alpha}\colon \dot{B}_{\frac{d}{d-\alpha},1}^{0,1} \to \dot{B}_{\frac{d}{d-\beta},1}^{0,1},\qquad \alpha < \beta \leq d,
}
which in our setting immediately follows from the definition~\eqref{BesovLorentzFormula}\textup, it suffices to consider the case~$\alpha < \frac{d}{2}$ \textup(equivalently\textup,~$p < 2$\textup) in Theorem~\textup{\ref{Main}}. What is more\textup, Theorem~\textup{\ref{Main}} allows the endpoint~$\alpha = d$ in the sense that if there are no charges~$a\otimes \delta_0$ in~$\WW$\textup, then
\eq{
\sum\limits_{k\in\Z} A^{-dk} \|f*(\psi_k - \psi_{k-1})\|_{L_{\infty}} \lesssim \|f\|_{L_1},\quad f \in \WW.
}
\end{Rem}

In the light of Remark~\ref{WtildeRemark}, Theorem~\ref{OurSobolevEmbedding} is a particular case of Theorem~\ref{Main}. Note that Theorem~\ref{Main} is stronger, for example, it implies the classical Hardy inequality (going back to~\cite{HardyLittlewood1927}):
\eq{
\Big(\int\limits_{\R_+} \frac{|\hat{f}(\xi)|^2}{|\xi|}\,d\xi\Big)^{\frac12} \lesssim \|f\|_{\Hardy_1};
}
here~$\Hardy_1$ is the analytic Hardy class on the line that consists of complex-valued summable functions having their Fourier transforms supported on the positive semiaxis. Theorem~\ref{Main} leads to yet another corollary in the spirit of Hardy's inequality; in the inequality below the symbol~$\sigma_r$ denotes the~$(d-1)$-Hausdorff measure on the sphere~$\{\zeta \in \R^d\mid|\zeta| = r\}$.
\begin{Cor}\label{HardyFourierCor}
Assume~$d \geq 2$. Let~$\WW$ be a closed linear subspace of~$\Sw'(\R^d,\R^\ell)$ that is invariant under translations and dilations and does not contain the charges~$a\otimes \delta_0$\textup,~$a\in \R^\ell \setminus \{0\}$. The inequality
\eq{\label{HardyFourier}
\sum\limits_{k\in\Z} A^{(1-d)k}\!\!\!\!\sup\limits_{r\in [A^k,A^{k+1})} \int\limits_{|\zeta| = r} |\hat{f}(\zeta)|\,d\sigma_r(\zeta) \lesssim \|f\|_{L_1},
}
holds true for any~$f\in\WW\cap L_1$. 
\end{Cor}
\begin{proof}
It suffices to prove the estimate
\eq{
A^{(1-d)k}\!\!\!\!\sup\limits_{r\in [A^{k-1},A^{k})} \int\limits_{|\zeta| = r} |\hat{f}(\zeta)|\,d\sigma_r(\zeta) \lesssim  A^{-\alpha k}\|f*(\psi_k - \psi_{k-1})\|_{L_{\frac{d}{d-\alpha}}},
}
where~$\alpha$ is a positive sufficiently small auxiliary parameter; then~\eqref{HardyFourier} will follow from Theorem~\ref{Main}. By dilation invariance, we may consider the case~$k=0$ only. Let us assume~$\hat{\psi}_0 - \hat{\psi}_{-1}$ is non-zero on the region where~$|\zeta|\in [A^{-1},1)$ (we may choose the function~$\psi$ in the definition of Besov space that satisfies this assumption). Thus, it suffices to show
\eq{
\|\hat{g}\|_{L_1(S_{r}(0))} \lesssim \|g\|_{L_{\frac{d}{d-\alpha}}}; \qquad S_r(0) = \Set{\zeta \in \R^d}{|\zeta| = r},\ r \in [A^{-1},1).
}
Since~$\alpha$ is close to zero, the summability exponent~$\frac{d}{d-\alpha}$ is close to one, and the desired inequality is a consequence of the Tomas--Stein theorem.
\end{proof}
The above corollary, in particular, leads to the inequality
\eq{
\int\limits_{\R^d}\frac{|\hat{f}(\xi)|}{|\xi|^{d-1}}\,d\xi \lesssim \|f\|_{\dot{W}_1^1}.
}
The latter inequality was first proved by Bourgain in an unpublished preprint~\cite{Bourgain1981}, see~\cite{PelczynskiWojciechowski1993} as well. For the case of the first gradient (as in Example~\ref{GagliardoNirenbergExample}), Corollary~\ref{HardyFourierCor} was proved by Kolyada in~\cite{Kolyada2019} for~$d \geq 3$ (the case~$d=2$ was open even in the case of the first gradient).

Before we pass to the proof of Theorem~\ref{Main}, we give a short description of the idea. The paper~\cite{ASW2018} suggested a discrete model for the problems mentioned above: the spaces~$\BV^\Omega$ have relatives in the world of discrete time martingales over regular filtrations. In the discrete model, the problems are simpler, and the paper~\cite{ASW2018} contains solutions to them. The approach is based upon four ingredients: the monotonicity formula (in the world of discrete martingales it reduces to a simple form of convexity in~$L_p$), the splitting into convex and flat atoms, an improvement of the monotonicity formula in the case the corresponding cancellation condition holds true, and a combinatorial argument. 

Our plan is to transfer the approach of~\cite{ASW2018} to the Euclidean setting, finding appropriate translations for the notions of a martingale, an atom, the monotonicity formula, and other objects of that paper. The translation appeared to be not quite literal, so there will be several new essences (the main are the horizontal graphs in Section~\ref{S6} below, there was no horizontal interaction in the discrete world; this also forces us to introduce another classification of atoms: there will be saturated and non-saturated atoms). However, the paper~\cite{ASW2018} might serve as a reading guide to the present text.

Section~\ref{S2} contains our interpretation of the words "martingale" and "monotonicity formula". The first notion is replaced with the heat extension; more specifically, we will consider discrete time martingale~$\{\Heat[f](\fdot,A^{-2k})\}_k$, where~$\Heat[f] = \Heat[f](x,t)$ is the heat extension of~$f$ and~$A$ is an extremely large number specified in Section~\ref{S7} (we will almost avoid the probabilistic terminology during the proof, however, the probabilistic point of view seems to be quite intuitive here). As for the monotonicity formula, we will be using a very particular case of a much more general monotonicity formula obtained by Bennett--Carbery--Tao in~\cite{BCT2006} (we will need to generalize this simple case to fit a weighted setting). The monotonicity formula is stated in Proposition~\ref{Karamata}.

Section~\ref{S3} provides an improvement of the Bennett--Carbery--Tao monotonicity formula for rank-one measures in~$\WW$ when the latter space does not contain delta measures. The idea is that the inequality expressing the monotonicity formula turns into equality only when~$f$ is a delta measure. Theorem~\ref{Robust} says that the measures~$\mu$ such that~$a\otimes \mu \in \WW$ are somehow separated from delta-measures, and thus, one may improve the monotonicity formula in Proposition~\ref{Karamata} for these measures~$\mu$. The separation is expressed through the notion of an invariant cone of measures from~\cite{Preiss1987} (we adjust this notion to fit our Schwartz class approach). Though we do not use the notion of tangent cone or tangent measure, the material of this section is reminiscent of some parts of the paper~\cite{Preiss1987}.

Section~\ref{S4} contains our interpretation of the notion "atom". We use a version of a time-frequency decomposition. However, we do not decompose the function over the space, but rather consider its norms in weighted~$L_1$-spaces, where a weight is localized in a neighborhood of an atom. The Uncertainty Principle says that the function~$f_k = \Heat[f](\fdot, A^{-2k})$ behaves like a function on the lattice~$A^{-k}\Z^d$. We exploit this principle by considering the values~$\|f_k\|_{L_1(w_{k,j})}$, where the weight~$w_{k,j}$ is concentrated in a neighborhood of the point~$A^{-k}j$,~$j \in\Z^d$; the quantity~$\|f_k\|_{L_1(w_{k,j})}$ is then treated as the value of the martingale~$f$ on the atom~$(k,j)$. We define convex and flat atoms similar to~\cite{ASW2018} and prove several useful lemmas about our weights. This allows us to estimate the sum over convex atoms in the manner similar to~\cite{ASW2018}; this estimate is proved in Proposition~\ref{ConvexControl}.

Section~\ref{S5} suggests a compactness argument that allows to perturb the improvement of the Bennett--Carbery--Tao monotonicity formula obtained in Section~\ref{S3}, i.e. to prove a similar monotonicity formula for functions~$f \in \WW$ that are somehow close to rank-one measures; the precise formulation is in Theorem~\ref{Compactness}. In~\cite{ASW2018}, it was shown that the growth of the~$L_p$-norm of a martingale on a flat atom is smaller than the growth of the same quantity for the martingale generated by a delta measure. In the Euclidean case, the situation is more complicated since we do not have perfect localization in the space variable (due to the Uncertainty Principle). So, we introduce an additional concentration assumption. We prove that if this assumption holds for some atom, and the atom is flat, then the function~$f$ is close to being a positive rank-one measure, and its~$L_p$ norm in a certain weighted space grows slower than that of a delta measure. 

The graphs in our reasonings will indicate subordination of atoms: if there is an arrow from~$\mathfrak{A}$ to~$\mathfrak{B}$, then some quantity of~$\mathfrak{B}$ may be estimated by a similar quantity of~$\mathfrak{A}$ uniformly.  Section~\ref{S6} introduces the graphs that mark the horizontal subordination of atoms. We study simple combinatorial properties of these graphs and introduce the second classification of atoms (the first being flat/convex). The atoms may be saturated and non-saturated. For saturated atoms, we prove good bounds for the growth of the~$L_p$-norm (this follows from Theorem~\ref{Compactness} since saturated atoms fulfill the concentration condition in that theorem). What is more, if a non-saturated atom is subordinate to a saturated one, then there is a certain control of the growth of the~$L_p$-norm as the weight drifts from the former atom to the latter one; the precise formulation is in Theorem~\ref{InductionStep}. 

The final Section~\ref{S7} concludes the proof. We introduce the graph~$\Gamma$ similar to the graph in~\cite{ASW2018}. Here, its structure is more complicated, it is not uniform. In fact, there might be vertices with infinitely many kids. However,~$\Gamma$ is still a forest, i.e. a union of trees. The improved monotonicity formula allows to run induction over an individual tree (Proposition~\ref{LebesgueInd}). A combinatorial argument then leads to the estimate of the sum over flat atoms by the sum over convex atoms similar to~\cite{ASW2018}.

\medskip

I would like to express my gratitude to Rami Ayoush and Michal Wojciechowski for long and fruitful collaboration and sharing their ideas with me.  
I also wish to thank Daniel Spector for discussions concerning this work, which, in particular lead to finding a mistake in an early version of the paper. 

 
\section{Gaussians and monotonicity formulas}\label{S2}
Consider the heat extension of a function~$f\in L_1(\R^d,\R^\ell)$, that is
\eq{
\Heat[f](x,t) = (4\pi t)^{-\frac{d}{2}}\int\limits_{\R^d}f(y)e^{-\frac{|x-y|^2}{4t}}\,dy,\qquad x\in \R^d,\ t > 0.
}
One may extend this definition to the case~$f\in \Sw'(\R^d,\R^\ell)$ in the usual way. This extension satisfies the heat equation
\eq{
(\Heat[f])_t = \Delta_x \Heat[f]
}
and the semigroup property
\eq{\label{Semigroup}
\Heat\big[\Heat[f](\fdot,s)\big](x,t) = \Heat[f](x,t+s),\qquad t,s > 0,\ x\in\R^d.
}
What is more, the operator~$f\mapsto \Heat[f](\fdot,t)$ is a Fourier multiplier with the symbol~$e^{-4\pi^2 t|\xi|^2}$, that is
\eq{
\mathcal{F}\Big[\Heat[f](\fdot,t)\Big](\xi) = e^{-4\pi^2 t|\xi|^2}\hat{f}(\xi).
}
Therefore,~$\Heat[f](\fdot,t) \in \WW$ for any~$t > 0$ provided~$f\in \WW$.

Let~$A$ be a large natural parameter to be chosen later. It will be convenient for further considerations to assume~$A$ is odd. 
We will use the notation
\eq{\label{RelationsParameters}
p = \frac{d}{d-\alpha},\quad \alpha = d\,\frac{p-1}{p},
}
which links the parameters in Theorem~\ref{Main}. By~$p'$ we mean the conjugate exponent~$p' = \frac{p}{p-1}$. 
\begin{Rem}
Theorem~\textup{\ref{Main}} follows from the inequality
\eq{\label{Besov1}
\sum\limits_{k\in\Z} A^{-\alpha k}\|\Heat[f](\fdot, A^{-2k})\|_{L_{p,1}} \lesssim \|f\|_{L_1},\qquad f\in \WW,
}
where the parameter~$A$ is our choice \textup(any suffices\textup) and the parameters~$\alpha$ and~$p$ satisfy~\eqref{RelationsParameters}. Indeed\textup, the inequality
\eq{
\|\I_{\alpha}[f]*(\psi_k - \psi_{k-1})\|_{L_{p,1}} \lesssim A^{-\alpha k} \|\Heat[f](\fdot, A^{-2k})\|_{L_{p,1}}
}
follows from the fact that the Fourier transform of the function
\eq{
\frac{\hat{\psi}(A^{-k}\xi) - \hat{\psi}(A^{-k+1}\xi)}{A^{-\alpha k}|\xi|^\alpha e^{-4\pi^2 A^{-2k}|\xi|^2}}
}
has uniformly bounded \textup(with respect to~$k$\textup)~$L_1$-norm\textup; by~\eqref{BesovLorentzFormula}\textup, summation of these inequalities for all~$k\in\Z$ leads to
\eq{
\|\I_{\alpha}[f]\|_{\dot{B}_{p,1}^{0,1}} \lesssim \sum\limits_{k\in\Z} A^{-\alpha k}\|\Heat[f](\fdot, A^{-2k})\|_{L_{p,1}}.
}
\end{Rem}
\begin{Rem}\label{ALesOne}
Using dilation invariance of the problem\textup, one may reduce~\eqref{Besov1} to
\eq{
\sum\limits_{k\geq 0} A^{-\alpha k}\|\Heat[f](\fdot, A^{-2k})\|_{L_{p,1}} \lesssim \|f\|_{L_1},\qquad f\in\WW,
}
or\textup, with the notation~$f_k(x) = \Heat[f](x, A^{-2k})$\textup,
\eq{\label{Besov2}
\sum\limits_{k\geq 0} A^{-\alpha k}\|f_k\|_{L_{p,1}} \lesssim \|f\|_{L_1},\qquad f\in\WW.
}
\end{Rem}

\begin{Def}
By a weight we mean a locally summable almost everywhere non-negative function~$w$ that defines a tempered distribution \textup(i.e. there exists~$M \in \N$ such that $\int_{B_R(0)} w(x)\,dx \lesssim R^M$ for all~$R > 1$\textup; here and in what follows\textup,~$B_r(x)$ stands for the closed Euclidean ball of radius~$r$ centered at~$x$\textup). We define the~$L_p(w)$ norm as
\eq{
\|f\|_{L_p(w)} = \Big(\int\limits_{\R^d}|f(x)|^pw(x)\,dx\Big)^\frac1p,
}
where the function~$f$ might be vector-valued.
\end{Def}
\begin{Le}\label{FirstMonotonicityLem}
Let~$w$ be a weight\textup, let~$g\in L_{1,\loc}\cap \Sw'(\R^d,\R^\ell)$ be a function\textup, and let~$p \geq 1$. Then\textup, for any~$t > 0$,
\eq{\label{FirstMonotonicity}
\|\Heat[g](\fdot,t)\|_{L_p(w)} \leq \|g\|_{L_p(\Heat[w](\fdot,t))},
}
provided the right hand side is finite.
\end{Le}
\begin{proof}
We raise the inequality to the power~$p$, use Jensen's inequality, and the Fubini theorem:
\eq{\label{MonotonicityProof}
\int\limits_{\R^d} |\Heat[g](x,t)|^p w(x)\,dx \leq \int\limits_{\R^d}\Heat[|g|^p](x,t)w(x)\,dx = \int\limits_{\R^d}|g|^p(x)\Heat[w](x,t)\,dx.
}
\end{proof}
\begin{Cor}\label{UnweightedMonotonicityCor}
For any~$m \geq k$\textup, 
\eq{
\|f_k\|_{L_p}\leq \|f_m\|_{L_p},\quad m \geq k,
}
provided the right hand side is finite.
\end{Cor}
\begin{proof}
By the semigroup property~\eqref{Semigroup} of the heat extension,
\eq{\label{fkfm}
f_k(x) = \Heat[f_m](x,A^{-2k} - A^{-2m}),
}
so the desired inequality indeed follows from Lemma~\ref{FirstMonotonicityLem} with~$w = 1$, since~$\Heat[1](x,t) = 1$ for any~$x$ and~$t$.
\end{proof}
\begin{Le}\label{JensenEquality}
Let~$p=1$. If~\eqref{FirstMonotonicity} turns into equality for some~$t > 0$\textup, then~$g = a\otimes h$\textup, where~$a\in \R^\ell$ and~$h$ is a non-negative scalar-valued function.
\end{Le}
\begin{proof}
It follows from formula~\eqref{MonotonicityProof} that
\eq{\label{UsefulFormula}
 \|g\|_{L_1(\Heat[w](\fdot,t))} - \|\Heat[g](\fdot,t)\|_{L_1(w)} = \int\limits_{\R^d}\Big(\Heat[|g|](x,t) - \big|\Heat[g](x,t)\big|\Big)w(x)\,dx.
}
Therefore, if~\eqref{FirstMonotonicity} turns into equality, then~$\Heat[|g|](x,t) = |\Heat[g](x,t)|$ for all~$x \in \R^d$. Plugging~$x= 0$, we get 
\eq{
\int |g|\,d\mu = \Big|\int g\,d\mu\Big|,
}
where~$\mu$ is a Lebesgue continuous measure with Gaussian density. Let~$a = \int g\,d\mu$. Recall that~$\pi_a$ denotes the orthogonal projection onto the vector~$a$. We write the chain of inequalities
\eq{
\int |g|\,d\mu \geq \int|\pi_a[g]|\,d\mu \geq \Big|\int \pi_a[g]\,d\mu\Big| = |\pi_a[a]| = |a|.
}
Both inequalities in this chain turn into equalities. Since we have~$|\pi_a[g]| = |g|$ almost everywhere, $g(x)$ is proportional to~$a$ for almost all~$x$. This is equivalent to~$g = a\otimes h$, where~$h$ is a scalar function. Then, since the second inequality in the chain turns into equality,~$h \geq 0$ almost everywhere. 
\end{proof}

\begin{St}\label{Karamata}
Let~$\mu$ be a non-negative scalar measure of tempered growth\textup, let~$w$ be a continuous weight. Then\textup,
\eq{\label{KaramataFormula}
\Big\|\Heat[\mu](\fdot,t)\Big\|_{L_p(\Heat[w](\fdot,\frac{1-t}{p}))} \leq t^{-\frac{d}{2}\frac{p-1}{p}} \|\Heat[\mu](\fdot,1)\|_{L_p(w)},\qquad t \in (0,1],
}
provided the right hand side is finite.
\end{St}
The unweighted case of this proposition was proved in~\cite{BCT2006} as a particular case of a more general monotonicity formula (see Proposition $5$ in the blog post~\cite{TaoBlogPost2019}). Till the end of this section, we will be using the notation
\mlt{\label{QpDef}
u(x,t) = \Heat[\mu](x,t),\quad v(x,t) = \Heat[w]\Big(x,\frac{1-t}{p}\Big), \\ Q_p[\mu,w](t) = t^{\frac{d(p-1)}{2}}\int\limits_{\R^d} u^p(x,t)v(x,t)\,dx, \quad x\in \R^d, t \in (0,1].
}
We will often suppress the first two arguments of~$Q_p$ if this does not lead to ambiguity. Note that~$u$ solves the classical heat equation, while~$v$ solves the rescaled backwards heat equation
\eq{\label{BackwardsHeat}
v_t = -\frac1p\Delta_x v.
}
Proposition~\ref{Karamata} will follow from the inequality
\eq{
\frac{\partial Q_p(t)}{\partial t} \geq 0,\qquad t\in(0,1).
}
This inequality is a consequence of the wonderful identity
\eq{\label{BCT}
\frac{\partial Q_p(t)}{\partial t} = \frac{p-1}{4} (4\pi)^{-\frac{dp}{2}} t^{-\frac{d}{2} -2}\int\limits_{\R^d}\Disp(x- Y_x)(\mu_{x,t}(\R^d))^pv(x,t)\,dx.
}
The formula needs clarifications. The measures~$\mu_{x,t}$ are given by
\eq{
d\mu_{x,t}(y) = e^{-\frac{|x-y|^2}{4t}}\,d\mu(y).
}
One may treat~$\R^d$ as a probability space equipped with the probability measure~$\frac{\mu_{x,t}}{\mu_{x,t}(\R^d)}$. The symbol~$Y_x$ then denotes the vectorial random variable~$y$ on this probability space. Note that since~$x$ is a constant random variable, we may replace~$\Disp(x-Y_x)$ with~$\Disp\, Y_x$ in~\eqref{BCT}. We use probabilistic language to formulate the result in a similar manner to the original paper~\cite{BCT2006}. In fact, we will not use the probabilistic terminology anymore.
\begin{proof}[Proof of formula~\eqref{BCT}]
Without loss of generality, we may assume that~$\mu$ is compactly supported, since the general case may be reduced to this one by a standard limiting argument. Then, the functions~$u$ and~$v$ are rapidly decaying at infinity, which allows integration by parts in the space variables without care about substitutions at the boundary. 

We start with a direct differentiation of~$Q_p$ using the definition of this function (we suppress the arguments of functions):
\eq{\label{DirectDiff}
\frac{\partial Q_p(t)}{\partial t} = t^{\frac{d}{2}(p-1)}\int\limits_{\R^d}\Big(\frac{d(p-1)}{2t}u^pv + pu^{p-1} u_tv + u^pv_t\Big)\,dx.
}
We leave this formula for a while and rewrite the right hand side of~\eqref{BCT} in more classical terms. We start with~$\Disp(x- Y_x) = \E|x- Y_x|^2 - |\E(x-Y_x)|^2$ and compute these summands separately:
\alg{
\E|x-Y_x|^2 &= \frac{\int_{\R^d}|y-x|^2e^{-\frac{|x-y|^2}{4t}}\,d\mu(y)}{\int_{\R^d}e^{-\frac{|x-y|^2}{4t}}\,d\mu(y)} = \frac{4t^2 \tilde{u}_t}{\tilde{u}};\\
\big|\E(x-Y_x)\big|^2 &= \bigg|\frac{\int_{\R^d}(y-x)e^{-\frac{|x-y|^2}{4t}}\,d\mu(y)}{\int_{\R^d}e^{-\frac{|x-y|^2}{4t}}\,d\mu(y)}\bigg|^2 = \bigg|\frac{-2t\nabla_x\tilde{u}}{\tilde{u}}\bigg|^2 = 4t^2\frac{|\nabla_x\tilde{u}|^2}{\tilde{u}^2},
}
where~$\tilde{u}(x,t) = (4\pi t)^{\frac{d}{2}}u(x,t)$ or~$\tilde{u}(x,t) = \mu_{x,t}(\R^d)$. We plug these formulas into the right hand side of~\eqref{BCT} and obtain
\eq{
(p-1)(4\pi)^{-\frac{dp}{2}}t^{-\frac d2}\int\limits_{\R^d}\Big(\tilde u_t \tilde u^{p-1} - |\nabla_x\tilde u|^2 \tilde u^{p-2}\Big)v\,dx.
}
We express this quantity in terms of~$u$ with the help of the formula~$\tilde{u}_t = (4\pi t)^\frac d2(u_t + \frac{d}{2t}u)$:
\eq{
\hbox{R. H. S. of \eqref{BCT}} = (p-1)t^{\frac{d}{2}(p-1)}\int\limits_{\R^d}\Big(u_tu^{p-1} + \frac{d}{2t}u^p - |\nabla_x u|^2u^{p-2}\Big) v\,dx.
}
Recalling~\eqref{DirectDiff}, we are left with proving the identity
\eq{
\int\limits_{\R^d} \Big(\frac{d(p-1)}{2t} u^pv + pu^{p-1}u_tv + u^pv_t\Big)\,dx = (p-1)\int\limits_{\R^d}\Big(\frac{d}{2t}u^pv + u^{p-1}u_tv - |\nabla_x u|^2u^{p-2}v\Big)\,dx,
}
which is equivalent to
\eq{\label{SimplerFormula}
\int\limits_{\R^d}\big(u^{p-1}u_t v + u^p v_t\big)\,dx = - (p-1) \int\limits_{\R^d}|\nabla_x u|^2 u^{p-2}v\,dx. 
}
We use that~$u$ solves the heat equation and~$v$ solves~\eqref{BackwardsHeat} to rewrite the left hand side of~\eqref{SimplerFormula}:
\eq{\label{Previous}
\int\limits_{\R^d}\big(u^{p-1}u_t v + u^p v_t\big)\,dx = \int\limits_{\R^d}u^{p-1}v\Delta_x u  - \frac{1}{p}\int\limits_{\R^d}u^p\Delta_x v.
}
We use integration by parts several times to rewrite the right hand side of~\eqref{SimplerFormula} (the angle brackets below denote the scalar product in~$\R^d$):
\mlt{
-(p-1)\int\limits_{\R^d}|\nabla_x u|^2u^{p-2}v = -\int\limits_{\R^d}\scalprod{(p-1)u^{p-2}\nabla_x u}{v\nabla_x u} = \int\limits_{\R^d}u^{p-1}\mathrm{div}[v\nabla_x u] =\\ \int\limits_{\R^d}u^{p-1}v\Delta_xu + \int\limits_{\R^d}u^{p-1}\scalprod{\nabla_xu}{\nabla_x v} = \int\limits_{\R^d}u^{p-1}v\Delta_xu -\frac1p\int\limits_{\R^d}u^{p}\Delta_x v.
}
This coincides with the right hand side of~\eqref{Previous}. So,~\eqref{SimplerFormula} is established together with~\eqref{BCT}.
\end{proof}
\begin{Rem}\label{EqualityRemark}
We have proved Proposition~\textup{\ref{Karamata}}. It also follows from~\eqref{BCT} that the inequality~\eqref{KaramataFormula} turns into equality if and only if~$\mu = \delta_x$ for some~$x\in \R^d$. Indeed\textup,~$\Disp(x-Y_x) = \Disp(Y_x)= 0$ if and only if~$y$ is constant~$\mu_{x,t}$ almost everywhere\textup, which means~$\mu_{x,t}$ is a delta measure.
\end{Rem}
\begin{Rem}
One may interpret Proposition~\textup{\ref{Karamata}} as an averaged version of Harnack's inequality for the heat equation. 
\end{Rem}

\section{Improving the monotonicity formula}\label{S3}
\begin{Def}
Let~$w$ be a continuous positive weight. Define its smoothness function~$\s[w]\colon \mathbb{R}_+\to[1,\infty]$ by the formula
\eq{
s[w](\zeta) = \sup\Set{\frac{w(x)}{w(y)}}{|x-y| \leq \zeta,\ x,y\in\R^d}.
}
\end{Def}
\begin{Rem}
The function~$\s[w]$ does not decrease and equals one at zero.
\end{Rem}
\begin{Ex}\label{ExampleOfPolynom}
Let~$\theta > 0$. Then\textup,~$\s[(1+|\cdot|)^{-\theta}](\zeta) = (1+\zeta)^{\theta}$.
\end{Ex}
\begin{Le}\label{SmoothnessConv}
Let~$\Psi \geq 0$. Then\textup,~$\s[w*\Psi]\leq \s[w]$ pointwise.
\end{Le}
\begin{proof}
It suffices to prove
\eq{
w*\Psi(x) \leq \s[w](\zeta)w*\Psi(y), \qquad\hbox{provided}\ |x-y|\leq \zeta.
}
By definition, we have
\eq{
w*\Psi(x) = \int\limits_{\R^d}\Psi(z)w(x-z)\,dz \leq \int\limits_{\R^d}\Psi(z)\s[w](\zeta)w(y-z)\,dz = \s[w](\zeta)w*\Psi(y),\quad |x-y| \leq \zeta.
}
\end{proof}
\begin{Le}\label{SmoothnessDil}
Let~$t \in (0,1)$\textup, let~$w$ be a continuous weight. Then the dilated weight~$W$\textup, i.e.~$W(x) = w(tx)$\textup, satisfies the inequality~$\s[W] \leq \s[w]$ pointwise.
\end{Le}
\begin{proof}
We simply work with the definition and use the property that the smoothness function does not decrease:
\eq{
\s[W](\zeta) = \sup\limits_{|x-y|\leq \zeta}\frac{W(x)}{W(y)} = \sup\limits_{|x-y|\leq \zeta}\frac{w(tx)}{w(ty)} = \s[w](t\zeta)\leq \s[w](\zeta).
}
\end{proof}
The next definition is inspired by~\cite{Preiss1987}.
\begin{Def}
We call a subset~$\MM$ of the class~$\Sw'(\R^d)$ an invariant cone of measures provided it satisfies the following properties\textup:
\begin{enumerate}[1\textup)]
\item any element~$\mu \in\MM$ is a measure\textup, i.e. a non-negative distribution\textup;
\item the set~$\MM$ is closed in the topology of~$\Sw'(\R^d)$\textup;
\item the set~$\MM$ is dilation invariant\textup;
\item the set~$\MM$ is translation invariant\textup;
\item the set~$\MM$ is a cone in the sense that~$c\mu\in\MM$ provided~$c \geq 0$ and~$\mu \in \MM$. 
\end{enumerate}
\end{Def}
\begin{Ex}
Let~$q \in [1, d-1]$ be a natural number. Let~$\MM_q$ be the set of all non-negative distributions in~$\R^d$ that depend on at most~$q$ coordinates in the sense that for any~$\mu \in \MM_q$ there exists~$L\in G(d,q)$ such that~$\mu$ is preserved by shifts by elements of~$L^{\perp}$. It is easy to see that~$\MM_q$ is an invariant cone of measures.
\end{Ex}
\begin{Ex}
The set
\eq{\label{ConeFromSpace}
\MM^\WW = \Set{\mu \in \Sw'(\R^d)}{\mu \geq 0\hbox{ and } \exists\, a\in \R^\ell \setminus \{0\}\hbox{ such that } a\otimes\mu \in \WW}
}
is an invariant cone of measures. For the classical case considered in Example~\textup{\ref{GagliardoNirenbergExample},} the cone~$\MM^\WW$ in formula~\eqref{ConeFromSpace} coincides with the cone~$\MM_1$ introduced in the previous example\textup; the case of the divergence free space given in Example~\textup{\ref{DivFree}} leads to~$\MM_{d-1}$.
\end{Ex}
\begin{Th}\label{Robust}
Let~$\MM$ be an invariant cone of measures that does not contain~$\delta_0$. Let~$G$ be a continuous positive weight that satisfies the smoothness bound
\eq{\label{ReqOnWeight}
\s[G](\zeta) \leq C_G(1+|\zeta|)^{\theta_G},\quad \zeta \in \R_+.
}
Let~$p \in (1,\infty)$ be a fixed number. There exists a tiny constant~$\delta$\textup, whose choice depends on the parameters~$\MM,\theta_G,C_G$\textup, and~$p$\textup, but not on the particular choice of~$G$ and~$\mu\in \MM$ such that
\eq{\label{Improvement}
\|\Heat[\mu](\fdot,t)\|_{L_p(\Heat[G](\fdot,\frac{1-t}{p}))} \leq t^{-\frac{d}{2}\frac{p-1}{p} + \delta}\|\Heat[\mu](\fdot,1)\|_{L_p(G)},\quad t\in (0,1],
}
provided the value on the right hand side is finite.
\end{Th}
This theorem might be thought of as a quantification of Remark~\ref{EqualityRemark}. The proof is quite lengthy (though fairly straightforward) and occupies the whole remaining part of this section. First, we would like to get rid of time. We recall the function~$Q_p = Q_p[\mu,G]$ defined in~\eqref{QpDef}.
\begin{Le}\label{RobustRescaling}
The estimate~\eqref{Improvement} follows from the inequality
\eq{\label{RobustDerivative}
\frac{\partial Q_p}{\partial t}(1) \geq \delta p Q_p(1),
}
once it is established for all continuous weights~$G$ satisfying~\eqref{ReqOnWeight} and~$\mu \in \MM$ uniformly.
\end{Le}
Before we pass to the proof, we note that if~$Q_p(1)$ is finite, then~$Q_p(t)$ is finite for any~$t\in (0,1)$ by Proposition~\ref{Karamata}. Therefore, we will be always working with finite quantities in the proofs below. 
\begin{proof}
Assume~\eqref{RobustDerivative} holds with the constant~$\delta$ uniform with respect to~$\mu$ and~$G$. The estimate~\eqref{Improvement} may be rewritten in terms of~$Q_p$ as~$Q_p(t) \leq t^{p\delta}Q_p(1)$,~$t \in (0,1]$; it clearly follows from
\eq{\label{RobustRescalingWanted}
\frac{Q_p'(t)}{Q_p(t)} \geq \frac{p\delta}{t},\qquad t\in (0,1).
}
We fix~$\mu$ and~$G$ and construct the functions~$u$ and~$v$ from them as prescribed by formula~\eqref{QpDef} (we set~$w:= G$ in that formula). We also consider dilated functions:
\eq{\label{DilationOfFunctions}
\tilde{u}(x,\theta) = u(tx,t^2\theta);\qquad \tilde{v}(x,\theta) = v(tx,t^2\theta),\qquad x\in \R^d,\ \theta > 0.
}
These functions solve the same heat and backwards heat equations as~$u$ and~$v$ correspondingly. Let us investigate the measure~$\tilde{\mu}$ and weight~$\tilde{G}$ that generate these functions as prescribed by formulas~\eqref{QpDef}.

The situation with~$\tilde{\mu}$ is simpler:~$\tilde{\mu}$ is the limit of~$\tilde{u}(\fdot,\theta)$ as~$\theta\to 0$, that is~$\tilde{\mu}$ is a dilation of~$\mu$. By the dilation invariance of~$\MM$, we have~$\tilde{\mu}\in \MM$.

The weight~$\tilde{G}$ is the limit value of the function~$\tilde{v}$ at the level~$\theta = 1$, and
\eq{
\tilde v(x,\theta) = v(tx,t^2\theta) = \Heat[G]\big(tx,\frac{1-t^2\theta}{p}\big).
}
Consequently, 
\eq{
\tilde G(x) = \Heat[G]\big(tx,\frac{1-t^2}{p}\big).
}
So,~$\tilde G(x) = G*\Phi(tx)$, where~$\Phi$ is a certain Gaussian. Recall that~$t < 1$, so Lemmas~\ref{SmoothnessConv} and~\ref{SmoothnessDil} lead to
\eq{
\s[\tilde G](\zeta) \leq \s[G](\zeta) \stackrel{\scriptscriptstyle \eqref{ReqOnWeight}}{\leq} C_G(1+|\zeta|)^{\theta_G}.
}
Therefore, we are allowed to apply our assumption~\eqref{RobustDerivative} to~$\tilde{\mu}$ and~$\tilde G$:
\eq{\label{Ineqtheta1}
Q_p'[\tilde \mu, \tilde G](1) \geq p\delta Q_p[\tilde \mu, \tilde G](1).
} 
It remains to express~$Q_p[\tilde \mu,\tilde G]$ in terms of~$Q_p[\mu,G]$:
\mlt{\label{TimeRescaling}
Q_p[\tilde \mu,\tilde G](\theta) = \theta^{\frac{d(p-1)}{2}}\int\limits_{\R^d} \tilde u^p(x,\theta)\tilde v(x,\theta)\,dx = \theta^{\frac{d(p-1)}{2}}\int\limits_{\R^d}u^p(tx,t^2\theta)v(tx,t^2\theta)\,dx = \\
\theta^{\frac{d(p-1)}{2}}t^{-d}\int\limits_{\R^d}u^p(y,t^2\theta)v(y,t^2\theta)\,dy = t^{-dp}Q_p[\mu,G](t^2\theta).
}
Plugging~$\theta = 1$, we get
\eq{\label{RobustValue}
Q_p[\mu,G](t^2) = t^{dp}Q_p[\tilde \mu,\tilde G](1).
}
If we differentiate~\eqref{TimeRescaling} with respect to~$\theta$ and plug~$\theta = 1$, we obtain
\eq{\label{RobustDerivative1}
Q_p'[\mu,G](t^2) = t^{dp-2}Q_p'[\tilde \mu,\tilde G](1).
}
A combination of~\eqref{Ineqtheta1},~\eqref{RobustValue},  and~\eqref{RobustDerivative1} leads to the desired estimate~\eqref{RobustRescalingWanted}.
\end{proof}
\begin{St}\label{RobustProp}
Let~$\rho$ be a fixed weight\textup:
\eq{\label{RhoFormula}
\rho(x) = (1+|x|)^{-\theta_G - 2d}.
} 
Let a positive weight~$G$ satisfy the smoothness bound~\eqref{ReqOnWeight}. For any~$\nu > 0$ there exists~$\eta > 0$ such that any tempered measure~$\mu \in \Sw'(\R^d)$ that satisfies 
\alg{
\label{Conc1}\int\limits_{\R^d}\Big(\int\limits_{\R^d} e^{-\frac{|x-y|^2}{4}}\,d\mu(y)\Big)^pG(x)\,dx = 1;\\
\label{Conc2}\int\limits_{\R^d}\Big(\int\limits_{\R^d} |\mass(x)-y|^2e^{-\frac{|x-y|^2}{4}}\,d\mu(y)\Big)\Big(\int\limits_{\R^d} e^{-\frac{|x-y|^2}{4}}\,d\mu(y)\Big)^{p-1}G(x)\,dx < \eta,
} 
where~$\mass(x)$ is short for
\eq{
\frac{\int_{\R^d} y\, e^{-\frac{|x-y|^2}{4}}\,d\mu(y)}{\int_{\R^d}e^{-\frac{|x-y|^2}{4}}\,d\mu(y)},
}
is concentrated around a point~$x_0\in \R^d$ in the sense that
\eq{\label{Conc3}
\nu\mu(B_\nu(x_0)) \geq\!\!\!\! \int\limits_{|x-x_0| \geq \nu}\!\!\!\! \rho(x-x_0)\,d\mu(x).
} 
The choice of~$\eta$ is independent of~$\mu$ and~$G$\textup, it depends on~$\nu$\textup,~$C_G$\textup, and~$\theta_G$ only.
\end{St}
\begin{Rem}
The replacement of~$m(x)$ with any other function of~$x$ in~\eqref{Conc2} will make this inequality stronger.  
\end{Rem}
\begin{proof}[Proof of Theorem~\ref{Robust} assuming Proposition~\ref{RobustProp}]
By Lemma~\ref{RobustRescaling}, it suffices to show~\eqref{RobustDerivative} with a certain uniformity in the choice of~$\delta$. Assume the contrary: let there exist a sequence~$\{\mu_n\}_n$ of measures in~$\MM$ and a sequence~$\{G_n\}_n$ of weights that satisfy~\eqref{ReqOnWeight} uniformly such that~$Q_p[\mu_n,G_n](1) = 1$ and~$Q_p'[\mu_n,G_n](1) \to 0$ as~$n\to \infty$.

By~\eqref{QpDef}, the pair~$(\mu_n,G_n)$ fulfills~\eqref{Conc1} and by~\eqref{BCT} it fulfills~\eqref{Conc2} with some~$\eta_n$ tending to zero (the assumption~$p > 1$ is crucial here, see formula~\eqref{BCT}). Thus, by Proposition~\ref{RobustProp} and translation invariance (we shift the measures to have~$x_0 = 0$), we may also assume that
\eq{\label{ConcentrationAtZero}
\nu_n\mu_n(B_{\nu_n}(0)) \geq \int\limits_{|x| \geq \nu_n}\rho(x)\,d\mu_n(x),
}
where~$\nu_n \to 0$ as~$n\to \infty$. Consider the measures~$\tilde \mu_n = \frac{\mu_n}{\mu_n(B_{\nu_n}(0))}$. Note that these measures still lie in~$\MM$. To get a contradiction, it suffices to show the limit relation
\eq{
\tilde\mu_n \stackrel{\scriptscriptstyle \Sw'(\R^d)}{\longrightarrow} \delta_0.
}

To verify the limit relation, pick~$f$ to be an arbitrary Schwartz function and rewrite the value of~$\tilde\mu_n$ as a functional at~$f$:
\eq{
\int\limits_{\R^d}f(x)\,d\tilde \mu_n(x) \stackrel{\scriptscriptstyle\eqref{ConcentrationAtZero}}{=} \frac{1}{\mu_n(B_{\nu_n}(0))}\int\limits_{B_{\nu_n}(0)}f(x)\,d\mu_n(x) + O(\nu_n)
}
since~$|f(x)| \lesssim \rho(x)$. It remains to notice that
\eq{
\frac{1}{\mu_n(B_{\nu_n}(0))}\int\limits_{B_{\nu_n}(0)}f(x)\,d\mu_n(x) \to f(0)
} 
since~$\nu_n\to 0$ and~$f$ is continuous. 
\end{proof}
Now we present the proof of Proposition~\ref{RobustProp}, thus completing the proof of Theorem~\ref{Robust}.
\begin{proof}[Proof of Proposition~\ref{RobustProp}]
We assume~$\nu < d^{-\frac12}$, which is no restriction. We split~$\R^d$ into the cubes~$Q_k$,~$k\in \Z^d$,
\eq{
Q_k = \prod_{i=1}^d[\nu k_i,\nu(k_i+1)),\quad k = (k_1,k_2,\ldots,k_d).
} 
Let also~$a_k = \mu(Q_k)$. We split the reasoning into four steps.

\paragraph{The quantity~$\sum_{k\in\Z^d}a_k^pG(\nu k)$ is separated away from zero and infinity.} To show the boundedness of the said sum, we start with a local estimate
\eq{\label{RPeq1}
\int\limits_{\R^d}e^{-\frac{|x-y|^2}{4}}\,d\mu(y) \geq \frac12 a_k,\quad x\in Q_k.
}
Therefore,
\mlt{
2^{-p}\sum\limits_{k\in\Z^d}a_k^pG(\nu k) \leq \sum\limits_{k\in\Z^d} 2^{-p} a_k^p\nu^{-d}\s[G](\nu\sqrt{d})\int\limits_{Q_k}G(x)\,dx \Leqref{RPeq1} \\ \nu^{-d}\s[G](\nu\sqrt{d})\sum\limits_{k\in\Z^d}\int\limits_{Q_k}\Big(\int\limits_{\R^d}e^{-\frac{|x-y|^2}{4}}\,d\mu(y)\Big)^pG(x)\,dx \Eeqref{Conc1} \nu^{-d}\s[G](\nu\sqrt{d}).
}
Thus, we have proved
\eq{\label{BoundednessFromAboveForWeightedSum}
\sum_{k\in\Z^d}a_k^pG(\nu k) \lesssim 1.
}

The reverse inequality is a bit harder to obtain. We start with yet another local (with respect to~$x$) estimate
\mlt{
\int\limits_{\R^d}e^{-\frac{|x-y|^2}{4}}\,d\mu(y) \lesssim \sum\limits_{k\in\Z^d}e^{-\frac{(|x-\nu k|- \sqrt{d})^2}{4}}a_k \leq\\ \Big(\sum\limits_{k\in\Z^d}e^{-\frac{(|x-\nu k|- \sqrt{d})^2}{4}}a_k^p\Big)^{\frac1p}\Big(\sum\limits_{k\in\Z^d}e^{-\frac{(|x-\nu k|- \sqrt{d})^2}{4}}\Big)^{\frac{1}{p'}}\lesssim \Big(\sum\limits_{k\in\Z^d}e^{-\frac{(|x-\nu k|-\sqrt{d})^2}{4}}a_k^p\Big)^{\frac1p}.
}
Consequently, by
\eq{
1\Lseqref{Conc1} \int\limits_{\R^d} \sum\limits_{k\in\Z^d}e^{-\frac{(|x-\nu k|-\sqrt{d})^2}{4}}a_k^p G(x)\,dx = \sum\limits_{k\in\Z^d}a_k^p\int\limits_{\R^d}e^{-\frac{(|x-\nu k| - \sqrt d)^2}{4}}G(x)\,dx,
}
the desired boundedness away from zero will follow once we verify the inequality
\eq{
\int\limits_{\R^d}e^{-\frac{(|x-\nu k| - \sqrt d)^2}{4}}G(x)\,dx \lesssim G(\nu k).
}
Here the verification is:
\mlt{
\int\limits_{\R^d}e^{-\frac{(|x-\nu k| - \sqrt d)^2}{4}}G(x)\,dx \leq G(\nu k) \int\limits_{\R^d}e^{-\frac{(|x-\nu k| - \sqrt d)^2}{4}}\s[G](|x-\nu k|)\,dx \Leqref{ReqOnWeight}\\ G(\nu k) C_G\int\limits_{\R^d} e^{-\frac{(|x|-\sqrt{d})^2}{4}}(1+|x|)^{\theta_G}\,dx \lesssim G(\nu k).
}
So, the proof of
\eq{\label{BoundFromBelow}
1\lesssim \sum_{k\in\Z^d}a_k^pG(\nu k)
}
is completed.

\paragraph{Kind points.} Let~$R$ be a large number to be specified later. Recall the weight~$\rho$ defined in~\eqref{RhoFormula}. A point~$k\in\Z^d$ is called \emph{kind} provided
\eq{
\nu^pa_k^p \geq \sum\limits_{\nu|k-m|\geq R}\rho(\nu(k-m))a_m^p.
}
We are going to show that most points are kind in the sense that
\eq{\label{RPKind}
\sum\limits_{k\, \hbox{\tiny is kind }} a_k^pG(\nu k) \geq \frac{1}{2}\sum\limits_{k\in\Z^d}a_k^pG(\nu k).
}
Note that both sides are finite by~\eqref{BoundednessFromAboveForWeightedSum}. A point that is not kind is \emph{evil}, then
\eq{
\sum\limits_{k\, \hbox{\tiny is evil }} a_k^pG(\nu k) \leq \nu^{-p}\!\!\!\!\! \sum\limits_{\genfrac{}{}{0pt}{-2}{k,m\in\Z^d}{\nu|k-m| \geq R}}\rho(\nu(k-m))a_m^pG(\nu k),
}
and~\eqref{RPKind} will follow provided we justify the estimate
\eq{
\nu^{-p}\!\!\!\!\!\sum\limits_{k\colon \nu|k-m|\geq R}\rho(\nu(k-m)) G(\nu k) \leq \frac12 G(\nu m),\quad \hbox{for any}\ m\in\Z^d.
}
Here the justification is:
\mlt{
\nu^{-p}\!\!\!\!\!\sum\limits_{k\colon \nu|k-m|\geq R}\rho(\nu(k-m)) G(\nu k) \leq \frac{G(\nu m)}{\nu^{p}} \sum\limits_{k\colon \nu|k-m|\geq R}\rho(\nu(k-m)) \s[G](\nu|k-m|) \LeqrefTwo{ReqOnWeight}{RhoFormula}\\
\frac{C_GG(\nu m)}{\nu^p}\!\!\!\!\!\sum\limits_{k\colon \nu|k-m|\geq R}(1+\nu|k-m|)^{-2d} \leq \frac{G(\nu m)}{2},
}
provided~$R$ is sufficiently large. We fix~$R$ to be so large that the latter estimate holds true together with
\eq{\label{add}
\sum\limits_{\nu|m| > R}\rho(\nu m) \leq (2\s[\rho](1))^{-p'},
}
where~$p'$ is the conjugate exponent to~$p$. Of course, the choice of~$R$ depends on~$\nu$. 

\paragraph{Good points.} Let~$\tau$ be a parameter to be chosen later. We call a point~$k\in \Z^d$ \emph{good} provided
\eq{\label{GoodDefinition}
\tau a_k \geq b_k,\quad \hbox{where}\qquad b_k = \min\Set{\!\!\!\!\!\sum\limits_{\genfrac{}{}{0pt}{-2}{m\colon\nu|k-m|\leq R,}{\sqrt{d} < |l-m|}}\!\!\!\!\! a_m}{l\in \Z^d}.
}
In other words,~$b_k$ is a sum of the~$a_m$ where~$m$ runs through the~$\nu^{-1} R$-neighborhood of~$k$ excluding some small ball (we exclude the points in the way that will make the sum as small as possible). We are going to prove that there exists a good kind point. More specifically, if~$\tau > \Theta(\nu,R)\eta$, then there is a point~$k$ that is~$R$-kind and~$\tau$-good. Here~$\Theta$ is a specific positive function of two positive arguments. It is high time to use condition~\eqref{Conc2}. 

We start with a local bound from below similar to~\eqref{RPeq1}. Let~$x\in Q_k$ and~$m(x) \in Q_l$ for some~$l \in \Z^d$. Then,
\mlt{
\int\limits_{\R^d}|\mass(x)-y|^2e^{-\frac{|x-y|^2}{4}}\,d\mu(y) \geq \sum\limits_{\genfrac{}{}{0pt}{-2}{m\colon \nu|k-m|\leq R}{\sqrt{d} < |l-m|}} \int\limits_{Q_m}|\mass(x)-y|^2e^{-\frac{|x-y|^2}{4}}\,d\mu(y) \geq\\ \nu^2 e^{-\frac{|R+\sqrt{d}|^2}{4}}\!\!\!\!\sum\limits_{\genfrac{}{}{0pt}{-2}{m\colon\nu|k-m|\leq R}{\sqrt{d} < |l-m|}} a_m,\quad x\in Q_k, \mass(x) \in Q_l.
}
Therefore,
\eq{
\int\limits_{\R^d}|\mass(x)-y|^2e^{-\frac{|x-y|^2}{4}}\,d\mu(y) \geq \nu^2 e^{-\frac{|R+\sqrt{d}|^2}{4}}b_k,\quad x\in Q_k,
}
which implies
\eq{\label{RPeq2}
\sum\limits_{k\in\Z^d} a_k^{p-1}b_kG(\nu k) \LeqrefTwo{Conc2}{RPeq1} \s[G](\nu\sqrt{d})\frac{2^{p-1}e^{\frac{|R+\sqrt{d}|^2}{4}}}{\nu^{2+d}}\eta.
}

We assume the contrary to our claim: let all kind points be~$\tau$-\emph{bad} (i.e. not~$\tau$-good). Then,
\mlt{
\tau \Lseqref{BoundFromBelow} \tau \sum\limits_{k\in\Z^d}a_k^pG(\nu k) \Leqref{RPKind} 2\tau \sum\limits_{k\,\hbox{\tiny is kind}}a_k^pG(\nu k) \stackrel{\hbox{\tiny kind points are bad}}{<}\\ 2\sum\limits_{k\in\Z^d} a_k^{p-1}b_kG(\nu k) \Leqref{RPeq2} \s[G](\nu\sqrt{d})\frac{2^{p}e^{\frac{|R+\sqrt{d}|^2}{4}}}{\nu^{2+d}}\eta.
}
We get a contradiction provided~$\tau :=2\Theta(\nu,R)\eta$, where~$\Theta$ is a specific positive function that may be easily written down (one needs to collect the constants in~\eqref{BoundFromBelow} and combine them with the above formula). Therefore, there exists a kind good point~$k_0$.

\paragraph{End of proof.} Note that if~$k$ is a good atom and~$\tau < 1$, then~$a_k$ must be excluded from the sum that defines~$b_k$ (because otherwise~$b_k \geq a_k$). In other words, if~$k$ is good, then the parameter~$l$ at which the minimum in~\eqref{GoodDefinition} is attained lies close to~$k$:~$|k-l| \leq \sqrt{d}$. Therefore, if~$k$ is good, then
\eq{\label{IfGood}
\tau a_k \geq\!\!\!\! \sum\limits_{\genfrac{}{}{0pt}{-2}{\nu|k-m|\leq R}{2\sqrt{d} < |k-m|}}\!\!\!\! a_m
}

We set~$x_0 = k_0$ and will show that such a choice indeed leads to~\eqref{Conc3} with~$\nu := 5\sqrt{d}\nu$ (we slightly enlarge this parameter). Without loss of generality, we may assume~$x_0 = 0$. We recall that~$\eta$ is still at our choice.

We wish to prove the inequality
\eq{\label{AlmostIt}
\nu a_0 \geq \int\limits_{|x|\geq 5\sqrt{d}\nu}\rho(x)\,d\mu(x).
}
We split the right hand side into two integrals to be estimated individually:
\eq{
\int\limits_{|x|\geq 5\sqrt{d}\nu} \leq \int\limits_{\genfrac{}{}{0pt}{-2}{\cup Q_m\colon 2\sqrt{d} < |m|}{\nu |m|<R}}+ \int\limits_{\cup Q_m\colon R \leq \nu|m|}.
}

The first part is estimated with the help of~\eqref{IfGood}:
\eq{
\sum\limits_{\genfrac{}{}{0pt}{-2}{\nu|m|<R}{2\sqrt{d}<|m|}}a_m \stackrel{\hbox{\tiny 0 is good}}{\leq} \tau a_0 < \frac{\nu}{2}a_0,
} 
provided~$\eta$ is sufficiently small (the specific bound may be expressed in terms of~$\Theta$).

The second part is estimated by
\eq{
\s[\rho](1) \sum\limits_{R < \nu |m|}a_m\rho(\nu m)\leq
 \s[\rho](1) \Big(\sum\limits_{R < \nu |m|}a_m^p\rho(\nu m)\Big)^\frac{1}{p}\Big(\sum\limits_{R < \nu |m|}\rho(\nu m)\Big)^\frac{1}{p'} \stackrel{\hbox{\tiny 0 is kind}}{\leq} \frac{\nu}{2}a_0,
}
since we also assumed~\eqref{add}. Thus,~\eqref{AlmostIt} is proved.

It remains to notice that~\eqref{AlmostIt} leads to~\eqref{Conc3} since~$\mu(B_{5\nu\sqrt{d}}(0)) \geq a_0$.
\end{proof}
We need to perturb Theorem~\ref{Robust} a little.
\begin{Cor}\label{RobustCorollary}
Let~$\MM$ be an invariant cone of measures that does not contain~$\delta_0$. Let~$p > 1$ be a fixed number. Let also the constants~$C_G$ and~$\theta_G$ be fixed. There exists~$\tilde \delta > 0$ such that for any sufficiently small~$t > 0$\textup, the following holds true. Let~$H$ be a solution of the heat equation on~$\R^d\times [t^2,1]$ such that~$H(\fdot,t^2) \in \MM$. Then\textup, the inequality
\eq{\label{eq350}
\|H(\fdot,t)\|_{L_p(\Heat[G](\fdot,\frac{1-t}{p}))} \leq t^{-\frac{d}{2}\frac{p-1}{p} + \tilde \delta} \|H(\fdot,1)\|_{L_p(G)}
}
is valid with any continuous positive weight~$G$ satisfying~\eqref{ReqOnWeight}\textup, provided the right hand side is finite.
\end{Cor}
\begin{proof}
Let us restate Theorem~\ref{Robust} in terms of the functions~$u$ and~$v$ generated by~\eqref{QpDef} with~$\mu$ and~$G$ in the roles of~$\mu$ and~$w$. It claims the inequality
\eq{
\|u(\fdot,t)\|_{L_p(v(\fdot, t))}\leq t^{-\frac{d}{2}\frac{p-1}{p} + \delta}\|u(\fdot,1)\|_{L_p(v(\fdot, 1))}, \quad t \in (0,1)
} 
for any functions~$u$ and~$v$ defined on the region~$\R^d\times (0,1)$,~$u$ solving the heat equation,~$v$ solving the backwards heat equation~\eqref{BackwardsHeat},  and such that~$u(\fdot,0) \in \MM$ and~$v(\fdot,1)$ satisfies~\eqref{ReqOnWeight}. We may freely shift the region~$\R^d \times (0,1)$ in any direction in~$\R^{d+1}$ without any changes to that statement. If we dilate the region~$\lambda$ times (making a change of variables in the style of~\eqref{DilationOfFunctions}), where~$\lambda$ is bounded away from zero and infinity to preserve~\eqref{ReqOnWeight} (with possibly slightly worse constant~$C_G$), we see that the inequality
\eq{
\|u(\fdot, t)\|_{L_p(v(\fdot,t))}\leq \Big(\frac{t-t_0}{t_1 - t_0}\Big)^{-\frac{d}{2}\frac{p-1}{p} + \delta}\|u(\fdot,t_1)\|_{L_p(v(\fdot, t_1))},\quad t\in (t_0,t_1),
}
holds true for any pair of functions~$u$ and~$v$ defined on the region~$\R^d\times [t_0, t_1]$, solving the same equations as usually, and such that~$u(\fdot,t_0) \in \MM$ and~$v(\fdot,t_1)$ satisfies~\eqref{ReqOnWeight}.

We plug~$t_0 := t^2$ and~$t_1 := 1$,~$u:= H$, and~$v(x,t):= \Heat[G](x,\frac{1-t}{p})$ into the latter statement to obtain
\eq{
\|H(\fdot,t)\|_{L_p(\Heat[G](\fdot,\frac{1-t}{p}))} \leq \Big(\frac{t-t^2}{1-t^2}\Big)^{-\frac{d}{2}\frac{p-1}{p} + \delta} \|H(\fdot,1)\|_{L_p(G)}.
}
Thus, to finish the proof, it remains to show
\eq{\label{eq353}
\Big(\frac{t-t^2}{1-t^2}\Big)^{-\frac{d}{2}\frac{p-1}{p} + \delta} \leq t^{-\frac{d}{2}\frac{p-1}{p} + \tilde \delta}
}
provided~$t$ is sufficiently small and we may choose arbitrarily small~$\tilde{\delta}$. We set~$\tilde \delta = \frac12 \delta$ and rewrite~\eqref{eq353} as
\eq{
(1+t)^{\frac{d}{2}\frac{p-1}{p} - \delta} \leq t^{-\frac12\delta}.
}
This inequality is true provided~$t$ is sufficiently small since the left hand side is continuous at zero while the right hand side blows up.
\end{proof}

\section{Time-frequency decomposition and control of convex atoms}\label{S4}
Recall that our main target is to prove~\eqref{Besov2}. We rewrite the right hand side as a telescopic sum
\eq{\label{telescope}
\|f\|_{W_\Omega^1} = \|f_0\|_{L_1} + \sum\limits_{k \geq 0}\big(\|f_{k+1}\|_{L_1} - \|f_k\|_{L_1}\big).
}
It is crucial that each summand is non-negative according to Corollary~\ref{UnweightedMonotonicityCor}. Due to technical reasons, we will be working with the sum
\eq{\label{Telescope3}
\sum\limits_{k \geq 0}\big(\|f_{k+3}\|_{L_1} - \|f_k\|_{L_1}\big),
}
which is bounded by~$3\|f\|_{L_1}$.

Let~$\theta_1 > d$ be a number we will specify later. Consider the weight~$w$ defined as
\eq{\label{Weightw}
w(x) = \frac{(1+|x|)^{-\theta_1}}{\sum\limits_{j\in\Z^d} (1+|x-j|)^{-\theta_1}}.
}
This weight satisfies the bounds
\eq{\label{BoundOnW}
c_w(1+|x|)^{-\theta_1} \leq w(x) \leq C_{w}(1+|x|)^{-\theta_1}, \quad x\in \R^d,
}
which, in particular, leads to~\eqref{ReqOnWeight} with~$\theta_G:= \theta_1$ and~$C_G := C_{w}/c_w$ (see Example~\ref{ExampleOfPolynom}). What is more, the shifts of~$w$ form a regular partition of unity:
\eq{\label{PartitionWeights}
\sum\limits_{j\in \Z^d}w(x-j) = 1,\quad x\in \R^d.
}

Consider now the partition of~$\R^d$ into~$A$-adic cubes. 
The cubes~$\{Q_{0,j}\}_{j}$ have centers in the lattice~$\Z^d$ and tile the whole space (up to a set of measure zero):
\eq{
Q_{0,j} = \Set{x\in \R^d}{|x-j|_{\ell_{\infty}^d} \leq \frac12},\qquad j\in \Z^d;
}
by the~$\ell_\infty^d$-norm we mean the standard~$\sup$-norm on~$\R^d$.
We form the collection~$\{Q_{k,j}\}_j$ dilating this family of cubes
\eq{
Q_{k,j} = \Set{x\in \R^d}{|A^kx - j|_{\ell_\infty^d} \leq \frac12},\qquad j\in \Z^d,\ k\geq 0.
}
Recall we assumed that~$A$ was odd. By virtue of this assumption, the family~$\{Q_{k,j}\}_{k,j}$ has a nice combinatorial property: any two cubes are either disjoint (up to a set of measure zero) or one of them contains the other.

We adjust the partition of unity~\eqref{PartitionWeights} to each scale:
\eq{
w_{k,j}(x) = w(A^kx - j), \qquad x\in\R^d,\ j\in \Z^d, k \geq 0.
}
These weights form a partition of unity for any fixed~$k$
\eq{
\sum\limits_{j\in\Z^d} w_{k,j} = 1,
}
and are regular in the sense that 
\eq{
c_w(1+|A^kx -j|)^{-\theta_1} \leq w_{k,j}(x) \leq C_{w}(1+|A^kx -j|)^{-\theta_1}, \qquad x\in \R^d;
}
in particular,
\eq{
\s[w_{k,j}](\zeta) \leq \frac{C_{w}}{c_w}(1+A^k\zeta)^{\theta_1},\qquad \zeta \geq 0.
}

The weights just introduced allow to split the sum~\eqref{Telescope3} further:
\eq{
\|f_{k+3}\|_{L_1} - \|f_k\|_{L_1} = \sum\limits_{j \in \Z^d}\Big(\|f_{k+3}\|_{L_1(\Heat[w_{k,j}](\fdot, A^{-2k} - A^{-2k-6}))} - \|f_k\|_{L_1(w_{k,j})}\Big)
}
since the weights~$\Heat[w_{k,j}](\fdot, A^{-2k} - A^{-2k-6})$ also form a partition of unity. By Lemma~\ref{FirstMonotonicityLem}, each summand in the above sum is non-negative. We have applied the said lemma with~$t := A^{-2k} - A^{-2k-6}$ and used~\eqref{fkfm}.
\begin{Def}
A pair~$(k,j)$\textup, where~$k \in \N\cup\{0\}$ and~$j \in \Z^d$\textup, is called an atom.
\end{Def}
Let~$\eps$ be a tiny parameter to be specified later.
\begin{Def}
An atom~$(k,j)$ is called~$\eps$-convex provided
\eq{
\|f_{k+3}\|_{L_1(\Heat[w_{k,j}](\fdot, A^{-2k} - A^{-2k-6}))} - \|f_k\|_{L_1(w_{k,j})} \geq \eps \|f_{k+3}\|_{L_1(\Heat[w_{k,j}](\fdot, A^{-2k} - A^{-2k-6}))}
}
and is called~$\eps$-flat in the case the above inequality is violated. The set of convex atoms is denoted by~$\Conv$\textup, the set of flat atoms is denoted by~$\Fl$.
\end{Def}
We will also say simply flat and convex suppressing the dependence on~$\eps$.
\begin{Rem}\label{AroundFlatness}
If~$(k,j)$ is~$\eps$-flat\textup, then
\mlt{
\|f_{k+2}\|_{L_1(\Heat[w_{k,j}](\fdot, A^{-2k} - A^{-2k-4}))} - \|f_k\|_{L_1(w_{k,j})} \Lref{\hbox{\tiny \textup{Lem.}}\scriptscriptstyle \ref{FirstMonotonicityLem}}\\ \|f_{k+3}\|_{L_1(\Heat[w_{k,j}](\fdot, A^{-2k} - A^{-2k-6}))} - \|f_k\|_{L_1(w_{k,j})} \leq\\ \frac{\eps}{1-\eps} \|f_k\|_{L_1(w_{k,j})} \Lref{\hbox{\tiny \textup{Lem.}}\scriptscriptstyle \ref{FirstMonotonicityLem}} \frac{\eps}{1-\eps}\|f_{k+2}\|_{L_1(\Heat[w_{k,j}](\fdot, A^{-2k} - A^{-2k-4}))}.
}
\end{Rem}
Convex atoms are easier to deal with, and we will soon prove "the half" of inequality~\eqref{Besov2} corresponding to convex atoms stated in the next proposition.
\begin{St}\label{ConvexControl}
For any~$\eps > 0$\textup, the inequality
\eq{
\sum\limits_{k\geq 0}A^{-\alpha k}\|f_k\|_{L_{p,1}(\!\!\!\!\bigcup\limits_{(k,j)\in \Conv}\!\!\!\! Q_{k,j})} \lesssim \|f\|_{L_1}
}
holds true. The constant in this inequality may depend on~$A, \eps,$ and~$\theta_1$.
\end{St}
The proof of Proposition~\ref{ConvexControl} needs some preparation, it is based upon three useful lemmas. The proof is situated at the end of the present section.
\begin{Le}\label{Split}
Let~$p,q\in [1,\infty)$ and let~$\{\Omega_j\}_j$ be a collection of measurable sets in~$\R^d$. Assume they all have non-zero measure and are disjoint up to sets of measure zero. Then\textup,
\eq{
\|g\|_{L_{p,q}(\cup_j \Omega_j)} \leq \sum\limits_j \|g\|_{L_{p,q}(\Omega_j)}
}
for any function~$g$ \textup($g$ may be vector-valued here\textup).
\end{Le}
Since we are going to prove a sharp inequality, we need to specify our choice of the Lorentz quasi-norm:
\eq{\label{LorentzDef}
\|h\|_{L_{p,q}(\Omega)} = p^{\frac{1}{q}} \Big\|t |\set{x\in \Omega}{|h(x)| \geq t}|^{\frac{1}{p}}\Big\|_{L_q(\R_+,\frac{dt}{t})},
}
where the absolute value of a set is its Lebesgue measure; here~$\Omega$ is a Borel set of positive measure. Note that the expression above is not necessarily a norm.
\begin{proof}[Proof of Lemma~\ref{Split}]
Fix some~$t > 0$. We start with the identity
\eq{
\Big|\Set{x\in \bigcup\limits_j\Omega_j}{|g(x)|\geq t}\Big| = \sum\limits_j |\set{x\in \Omega_j}{|g(x)|\geq t}|,
}
which implies the inequality
\eq{
\Big|\Set{x\in \bigcup\limits_j\Omega_j}{|g(x)|\geq t}\Big|^{\frac{1}{p}} \leq \sum\limits_j |\set{x\in \Omega_j}{|g(x)|\geq t}|^{\frac1p}.
}
It remains to recall the definition~\eqref{LorentzDef} and use the triangle inequality in~$L_q$.
\end{proof}
\begin{Le}\label{HeatingWeights}
Let~$G$ be a weight on~$\R^d$ that satisfies the estimates
\eq{\label{BiIneq1}
c_G(1+|x|)^{-\theta_G} \leq G(x) \leq C_{G}(1+|x|)^{-\theta_G}, \quad x\in \R^d,
}
for some~$\theta_G > d$. Then\textup, there are constants~$\tilde c_G$ and~$\Tilde C_G$ that do not depend on~$G$ itself\textup, but only on~$\theta_G, c_G$\textup, and~$C_G$\textup, such that the estimate
\eq{\label{BiIneq2}
\tilde c_G(1+|x|)^{-\theta_G} \leq \Heat[G](x,t) \leq \tilde C_{G}(1+|x|)^{-\theta_G}, \quad x\in \R^d,
}
holds true for any~$t\in [0,2]$.
\end{Le}
\begin{proof}
Let us first prove the estimate from below:
\mlt{
\Heat[G](x,t) = (4\pi t)^{-\frac{d}{2}}\int\limits_{\R^d}G(x-y)e^{-\frac{|y|^2}{4t}}\,dy \geq c_G (4\pi t)^{-\frac{d}{2}}\int\limits_{\R^d}(1+|x-y|)^{-\theta_G}e^{-\frac{|y|^2}{4t}}\,dy \geq\\
c_G (4\pi t)^{-\frac{d}{2}}\int\limits_{B_{\sqrt{t}}(0)}(1+|x-y|)^{-\theta_G}e^{-\frac{|y|^2}{4t}}\,dy \stackrel{\scriptscriptstyle t \leq 2}{\gtrsim} t^{-\frac d2} (1+|x|)^{-\theta_G}\!\!\!\!\int\limits_{B_{\sqrt{t}}(0)}\!\!\!dy \gtrsim (1+|x|)^{-\theta_G}.
}
Let us prove the estimate from above:
\mlt{
\Heat[G](x,t) = (4\pi t)^{-\frac{d}{2}}\int\limits_{\R^d}G(x-y)e^{-\frac{|y|^2}{4t}}\,dy \leq C_G (4\pi t)^{-\frac{d}{2}}\int\limits_{\R^d}(1+|x-y|)^{-\theta_G}e^{-\frac{|y|^2}{4t}}\,dy \stackrel{\scriptscriptstyle t \leq 2}{\leq} \\
C_G(4\pi t)^{-\frac d2} (1+|x|)^{-\theta_G}\Big(\s[(1+|\cdot|)^{-\theta_G}](\sqrt2)|B_{\sqrt{t}}(0)|+ \sum\limits_{k \geq 0}\s[(1+|\cdot|)^{-\theta_1}](2^{k+1}\sqrt{t})\!\!\!\!\!\!\!\!\!\!\int\limits_{2^k\sqrt{t} \leq |y| < 2^{k+1} \sqrt{t}}\!\!\!\!\!\!\!\!\!\! e^{-\frac{|y|^2}{4t}} \,dy\Big) \Lsref{{\hbox{\tiny Ex. \ref{ExampleOfPolynom}}}} \\
t^{-\frac d2} (1+|x|)^{-\theta_G}(t^{\frac d2}+ \sum\limits_{k \geq 0}  (1+2^{k+1}\sqrt{t})^{\theta_G} (2^{k+1}\sqrt{t})^d e^{-2^{2k-2}}\Big) \Lsref{\scriptscriptstyle t < 2} (1+|x|)^{-\theta_G}.
}
\end{proof}
\begin{Rem}\label{SeparateRemark}
As it follows from the proof\textup, the first inequality in~\eqref{BiIneq1} implies the first inequality in~\eqref{BiIneq2}\textup, whereas the second inequality in~\eqref{BiIneq1} implies the second in~\eqref{BiIneq2}.
\end{Rem}
\begin{Le}\label{L1Lp}
Let~$\theta_u$ and~$\theta_v$ be two constants larger than~$d$. Let the weights~$u$ and~$v$ satisfy the inequalities
\alg{
\label{uBound}u(x) \geq c_u(1+|x|)^{-\theta_u},\quad x\in \R^d;\\
v(x) \leq C_v(1+|x|)^{-\theta_v},\quad x\in \R^d.
}
Let~$p \in [1,2]$. Assume~$\theta_v \geq p\theta_u$. Then\textup, for any~$s\in [\frac12,2]$ and any~$f\in L_1(u)$\textup, the inequality
\eq{\label{L1LpFormula}
\|\Heat[f](\fdot,s)\|_{L_p(v)} \lesssim \|f\|_{L_1(u)}
}
holds true with the constant independent of~$s,u$\textup, and~$v$\textup; the constant\textup, however\textup, may depend on~$\theta_u,\theta_v, p, c_u$\textup, and~$C_v$.
\end{Le}
\begin{proof}
It suffices to test the inequality against~$f = \delta_{x_0}$. In this case, the right hand side of~\eqref{L1LpFormula} equals~$u(x_0)$, whereas the left hand side is
\mlt{
\|\Heat[\delta_{x_0}](\fdot,s)\|_{L_p(v)} = \Big(\int\limits_{\R^d}(4\pi s)^{-\frac{dp}{2}}e^{-\frac{p|x-x_0|^2}{4s}}v(x)\,dx\Big)^{\frac1p}\stackrel{\scriptscriptstyle \frac12 \leq s \leq 2}{\lesssim}\\ C_v^\frac{1}{p}\Big(\Heat[(1+|\cdot|)^{-\theta_v}]\big(x_0, \frac{s}{p}\big)\Big)^{\frac1p} \Lsref{\hbox{\tiny Rem. \ref{SeparateRemark}}} (1+|x_0|)^{-\frac{\theta_v}{p}}.
} 
Therefore, the inequality under consideration is reduced to~$(1+|x_0|)^{-\frac{\theta_v}{p}} \lesssim u(x_0)$, which is true by~\eqref{uBound} and the requirement~$p\theta_u\leq \theta_v$.
\end{proof}
We need to introduce the space~$L_{p,1}(v)$. The norm in this space is defined as
\eq{
\|h\|_{L_{p,1}(v)} = p^{} \Big\|\Big(\int\limits_{\Omega_t} v(x)\,dx\Big)^{\frac{1}{p}}\Big\|_{L_1(\R_+)},\qquad \Omega_t = \set{x\in \R^d}{|h(x)| \geq t},\ t > 0.
}
Note that this agrees with~\eqref{LorentzDef} if~$v = \chi_\Omega$ (usually we prefer to work with continuous weights, so we give two definitions).
\begin{Cor}
Standard interpolation theory implies that if~$\theta_v > p\theta_u$ in the notation and assumptions of the previous lemma\textup, then
 \eq{
\|\Heat[f](\fdot,s)\|_{L_{p,1}(v)} \lesssim \|f\|_{L_1(u)},\qquad s\in\big[\frac12,2\big].
}
\end{Cor}
Sometimes we will need to keep track of the constants in our inequalities. In fact, only the dependence on~$A$ is of crucial importance (there will be an exception in Section~\ref{S7} far below). We will write~$\lesssim_A$ to signify that the constant hidden inside the symbol~$\lesssim$ is independent of~$A$. To avoid ambiguity, we will usually comment on the said independence.  
\begin{Cor}\label{L1LpFormulaRescaled}
Let~$(k,j)$ be an atom. Let the weights~$u_{k,j}$ and~$v_{k,j}$ satisfy the bounds
\alg{
\label{uBoundRescaled} u_{k,j}(x) \geq c_u(1+|A^kx -j|)^{-\theta_u},\quad x\in \R^d;\\
v_{k,j}(x) \leq C_v(1+|A^kx -j|)^{-\theta_v},\quad x\in \R^d.
}
Let also~$\theta_v > p\theta_u$ and~$p \leq 2$. The inequality 
\eq{
A^{-\frac{d(p-1)}{p} k}\|\Heat[f](\fdot,s)\|_{L_{p,1}(v_{k,j})} \lesssim_A \|f\|_{L_1(u_{k,j})}
}
holds true whenever~$s\in [\frac 12 A^{-2k},2A^{-2k}]$. The constant in the inequality is uniform with respect to the parameters~$s, A, k, j, u, v,$ and~$f$\textup; however\textup, it might depend on~$p, \theta_u, \theta_v, c_u$\textup, and~$C_v$.
\end{Cor}
The last corollary will be needed only in Section~\ref{S6}. It is more convenient to present it here. Recall~$\alpha = \frac{d(p-1)}{p}$.
\begin{Cor}\label{EarlyCor}
Let~$p \leq 2$. The inequality
\eq{
\|f_1\|_{L_p(Q_{0,i})} \lesssim_A A^{\alpha} \|f_2\|_{L_1(\Heat[w_{0,i}](\fdot,1-A^{-4}))}
}
holds true for all~$f$\textup,~$i \in \Z^d$\textup, and all~$A > 2$ uniformly \textup(the constants may depend on~$\theta_1$\textup).
\end{Cor}
\begin{proof}
Without loss of generality, we may assume~$i = 0$. By Lemma~\ref{Split}, it suffices to show
\eq{\label{EarlyCorWanted}
\sum\limits_{j\colon Q_{1,j}\subset Q_{0,0}}\!\!\!\!\|f_1\|_{L_p(Q_{1,j})} \lesssim_A A^{\alpha} \|f_2\|_{L_1(\Heat[w_{0,0}](\fdot,1-A^{-4}))}.
}
By~\eqref{fkfm} and Corollary~\ref{L1LpFormulaRescaled} (with~$k=1$,~$s = A^{-2} - A^{-4}$, $v_{1,j}:= \chi_{Q_{1,j}}$,~$u_{1,j} := w_{1,j}$; so~$\theta_u = \theta_1$ and~$\theta_v = p\theta_1 + 1$),
\eq{
\sum\limits_{j\colon Q_{1,j}\subset Q_{0,0}}\!\!\!\!\|f_1\|_{L_p(Q_{1,j})} \lesssim_A A^{\alpha}\sum\limits_{j\colon Q_{1,j}\subset Q_{0,0}}\|f_2\|_{L_1(w_{1,j})}.
}
Thus, the desired inequality~\eqref{EarlyCorWanted} will follow with the help of Lemma~\ref{HeatingWeights}, provided we show that
\eq{\label{SumOfWeights}
\sum\limits_{j\colon Q_{1,j} \subset Q_{0,0}}w_{1,j}(x) \lesssim_A w_{0,0}(x),\quad x\in \R^d.
}

To show~\eqref{SumOfWeights}, we first note that the weights~$\{w_{1,j}\}_{j\in\Z^d}$ form a partition of unity, which, in particular, means that the left hand side of this inequality never exceeds one. Consequently, it suffices to prove the inequality in the case~$|x| \geq 2\sqrt{d}$. Note that in this case all the quantities~$w_{1,j}(x)$ are comparable since~$Q_{1,j} \subset Q_{0,0}$:
\eq{
w_{1,j}(x) \lesssim_A (1+|Ax - j|)^{-\theta_1} \lesssim_A (1+A|x|)^{-\theta_1},\quad A|x| \geq 2|j|.
}
Therefore,
\eq{
\sum\limits_{j\colon Q_{1,j} \subset Q_{0,0}}w_{1,j}(x) \lesssim_A A^{d} (1+A|x|)^{-\theta_1}\leq A^{d-\theta_1}|x|^{-\theta_1} \lesssim_A w_{0,0}(x),
}
provided~$A > 2$,~$\theta_1 > d$, and~$|x| \geq 2\sqrt{d}$.
\end{proof}
\begin{proof}[Proof of Proposition~\ref{ConvexControl}.]
We first apply Lemma~\ref{Split}:
\eq{\label{PrConvContr1}
\sum\limits_{k\geq 0}A^{-\alpha k}\|f_k\|_{L_{p,1}(\!\!\!\!\bigcup\limits_{(k,j)\in \Conv}\!\!\!\! Q_{k,j})} \leq \sum\limits_{k\geq 0}A^{-\alpha k}\sum\limits_{j\colon (k,j)\in \Conv} \|f_{k}\|_{L_{p,1}(Q_{k,j})}.
}
Then, we use the representation~$f_k = \Heat [f_{k+3}](\fdot, A^{-2k} - A^{-2k-6})$ (see~\eqref{fkfm}) and apply Corollary~\ref{L1LpFormulaRescaled} with~$u_{k,j} := \Heat[w_{k,j}](\fdot, A^{-2k} - A^{-2k-6})$ (by Lemma~\ref{HeatingWeights}, this weight satisfies~\eqref{uBoundRescaled} with~$\theta_u:= \theta_1$),~$v_{k,j}:= \chi_{Q_{k,j}}$ (so we set~$\theta_v = p\theta_u + 1$), and~$s = A^{-2k} - A^{-2k-6}$, which is fine when~$A \geq 2$:
\eq{
A^{-\alpha k}\|f_{k}\|_{L_{p,1}(Q_{k,j})} \lesssim \|f_{k+3}\|_{L_1(\Heat [w_{k,j}](\fdot, A^{-2k} - A^{-2k-6}))}.
}
We continue the estimate~\eqref{PrConvContr1} using the definition of a convex atom:
\mlt{
\sum\limits_{k\geq 0}\sum\limits_{j\colon (k,j)\in \Conv} \|f_{k+3}\|_{L_1(\Heat [w_{k,j}](\fdot, A^{-2k} - A^{-2k-6}))} \leq\\ \frac{1}{\eps} \sum\limits_{k \geq 0}\sum\limits_{j\in\Z^d} \Big(\|f_{k+3}\|_{L_1(\Heat [w_{k,j}](\fdot, A^{-2k} - A^{-2k-6}))} - \|f_k\|_{L_1(w_{k,j})}\Big) = \\
\frac{1}{\eps}\sum\limits_{k \geq 0}\Big(\|f_{k+3}\|_{L_1} - \|f_k\|_{L_1}\Big)\Lseqref{telescope} \|f\|_{L_1}.
}
\end{proof}

\section{Compactness argument}\label{S5}
The informal meaning of the flat/convex classification of atoms is that~$\eps$-flat atoms mark the places in~$\R^d$ where the function~$f_k$ is close to a positive rank-one function. Recall Lemma~\ref{JensenEquality}, which says that the presence of a~$0$-flat atom ensures~$f_{k+3} = a\otimes h$ with~$h \geq 0$. So, we are looking for a compactness argument which will allow us to replace~$0$ with~$\eps$, obtain that~$f_k$ is somehow close to being a positive rank-one function, and then apply Corollary~\ref{RobustCorollary} to it. Unfortunately, this principle seems to be false in general. Imagine a function~$f_k$ concentrated far way from the cube~$Q_{k,j}$ and obtaining arbitrary values on this cube. Seemingly, we can deduce nothing from the flatness of~$(k,j)$ since the behavior of~$f_k$ in a neighborhood of~$Q_{k,j}$ is not very much related to the weighted norms we consider. 

This hints us that our searched-for compactness argument should have an assumption that~$f_k$ is concentrated in a neighborhood of~$Q_{k,j}$. Fix some number~$\theta_2$ to be specified later and require~$\theta_2 < \theta_1$. Consider the weight
\eq{\label{uformula}
u(x) = (1+|x|)^{-\theta_2}.
}
Let also~$C$ be some fixed constant. We express the concentration of~$f_0$ on~$(0,0)$ by the inequality
\eq{\label{Concentrated}
\|f_2\|_{L_1(u)} \leq C\|f_2\|_{L_1(\Heat[w_{0,0}] (\fdot,1 -A^{-4}))}.
}
The presence of~$f_2$ instead of~$f_0$ in this inequality is needed for technical reasons. This condition might be transferred to an arbitrary atom in the usual way. Let also~$\theta_3 > d$ be a number to be specified later. Now we only require
\eq{
\theta_3 > p\theta_1.
} 
Consider the weight
\eq{
v(x) = (1+|x|)^{-\theta_3}.
}
Recall the cone~$\MM^\WW$ naturally related to the space~$\WW$ by formula~\eqref{ConeFromSpace}.
\begin{Th}\label{Compactness}
Let~$\delta_0 \notin \MM^\WW$. Let~$\tilde \delta$ be the specific number obtained in Corollary~\textup{\ref{RobustCorollary}} with the cone~$\MM^\WW$ and the weight~$G:= v$. Fix~$\theta_1, \theta_2,$ and~$\theta_3$. For any~$C$ fixed and any~$A$ sufficiently large there exists~$\eps$ depending on all the parameters \textup(except for the function~$f\in L_1\cap \WW$\textup) such that if~$(0,0)$ is~$\eps$-flat and~$f$ satisfies the concentration condition~\eqref{Concentrated}\textup, then
\eq{\label{ImprovedMonotonicityNonPositive}
\|f_1\|_{L_p(\Heat[v](\fdot,\frac{1-A^{-2}}{p}))} \leq A^{\frac{d(p-1)}{p} - \frac14\tilde \delta} \|f_0\|_{L_p(v)}.
}
Moreover\textup, the inequality
\eq{\label{CubeFromBelow}
\|f_2\|_{L_1(\Heat[w_{0,0}] (\fdot, 1-A^{-4}))} \lesssim_A\|f_0\|_{L_p(Q_{0,0})}
}
holds true provided~$\eps$ is sufficiently small \textup(the constant in this inequality does not depend on~$\eps$ as well as on~$A$ provided~$\eps$ is sufficiently small\textup).
\end{Th}
The proof of Theorem~\ref{Compactness} occupies the remaining part of this section. We start with several useful lemmas. Let~$\Omega \subset \R^d$ be an open convex set. We define the Lipschitz semi-norm by the formula
\eq{
\|g\|_{\Lip(\Omega)} = \sup\limits_{\genfrac{}{}{0pt}{-2}{x,y\in\Omega}{x\ne y}}\frac{|g(x) - g(y)|}{|x-y|};
}
note that~$\|g\|_{\Lip(\Omega)} = \|\nabla g\|_{L_{\infty}(\Omega)}$.
\begin{Le}\label{LipschitzLemma}
Let~$R > 0$ be a fixed number and let~$G$ be a weight satisfying the bound~\eqref{uBound} with the parameters~$\theta_G$ and~$c_G$. Then\textup,
\alg{
\label{Lip1}\|\Heat[f](\fdot,s)\|_{\Lip(B_R(0))} \lesssim \|f\|_{L_1(G)};\\
\label{Lip2}\|\Heat[|f|](\fdot,s)\|_{\Lip(B_R(0))} \lesssim \|f\|_{L_1(G)}
}
provided~$s\in [\frac12,2]$. The constants in these inequalities are independent of~$s$.
\end{Le}
\begin{proof}
We would like to estimate the quantities~$|\nabla \Heat[f](x,s)|$ and~$|\nabla \Heat[|f|](x,s)|$ when~$|x| < R$ and~$s \in [\frac12,2]$. Both these quantities do not exceed
\eq{
(4\pi s)^{-\frac d2}\int\limits_{\R^d}|f(y)||\nabla_xe^{-\frac{|x-y|^2}{4s}}|\,dy.
}
This integral is bounded by~$L\|f\|_{L_1(G)}$, where
\eq{
L \lesssim \sup\limits_{\genfrac{}{}{0pt}{-2}{|x| \leq R}{y\in\R^d}} \frac{|x-y|e^{-\frac{|x-y|^2}{4s}}}{G(y)} \Lref{\scriptscriptstyle s\leq 2} \sup\limits_{y\in\R^d}\frac{|y+R|e^{-\frac{(|y| - R)^2 - R^2}{8}}}{c_G(1+|y|)^{-\theta_G}}.
}
This quantity is finite for any choice of the parameters.
\end{proof}
\begin{Rem}\label{LipschitzRemark}
If we pick arbitrary~$s\in (0,1]$\textup, then the inequalities~\eqref{Lip1} and~\eqref{Lip2} still hold true\textup, however\textup, the constants in these inequalities are not uniform with respect to~$s$.
\end{Rem}
Till the end of this section, we will use the notation
\eq{\label{Tildew}
\tilde{w} = \Heat[w](\fdot,1-t),
}
where the weight~$w$ is given by~\eqref{Weightw}. 
\begin{Le}\label{AntiCancellationLemma}
Assume~$f$ satisfies the flatness condition in the form
\eq{\label{Milderflatness}
\|f\|_{L_1(\tilde{w})} - \|\Heat[f](\fdot,1-t)\|_{L_1(w)} \leq \eps \|f\|_{L_1(\tilde{w})},
}
here~$t \in [0,\frac12]$ is a fixed real. 
Fix a large number~$R$ and assume that
\eq{\label{SmallTail}
\int\limits_{|x|\geq R}|f(x)|\tilde{w}(x)\,dx \leq \frac12 \|f\|_{L_1(\tilde{w})}.
}
There exists~$c > 0$ such that the inequality
\eq{\label{Canc}
\big|\Heat[f](x,1-t)\big| \geq c\|f\|_{L_1(\tilde{w})}
}
holds true for any~$x\in B_R(0)$ provided~$\eps$ is sufficiently small. The constant~$c$ does not depend on~$t$ as long as~$\eps$ is sufficiently small and~$R$ is fixed.
\end{Le}
\begin{proof}
Let us first prove a similar inequality, which definitely avoids unwanted cancellations:
\eq{\label{NoCanc}
\Heat[|f|](x,1-t) \geq c_1\|f\|_{L_1(\tilde{w})},\quad |x| \leq R.
}
This is straightforward:
\mlt{
\Heat[|f|](x,1-t) = \big(4\pi(1-t)\big)^{-\frac d2}\int\limits_{\R^d} |f(y)| e^{-\frac{|x-y|^2}{4(1-t)}}\,dy \geq\\ \big(4\pi(1-t)\big)^{-\frac d2}\int\limits_{|y|\leq R} |f(y)| e^{-\frac{|x-y|^2}{4(1-t)}}\,dy \geq \tilde c\int\limits_{|y|\leq R}|f(y)|\tilde w(y),
}
where
\eq{
\tilde c = \inf\limits_{|x|,|y| \leq R} \frac{(4\pi)^{-\frac d2}(1-t)^{-\frac d2}e^{-\frac{|x-y|^2}{4(1-t)}}}{\tilde{w}(y)} \gtrsim e^{-R^2}(1+R)^{\theta_1}
}
(we have used Lemma~\ref{HeatingWeights} here). To finish the proof of~\eqref{NoCanc}, we simply use~\eqref{SmallTail} and set~$c_1:= \tilde{c}/2$.

Now we return to the proof of~\eqref{Canc}. We will show this inequality holds true with~$c := \frac12 c_1$. We recall~\eqref{UsefulFormula} to state that our flatness assumption~\eqref{Milderflatness} leads to
\eq{\label{UsefulFormulaRestated}
\int\limits_{\R^d}\Big(\Heat[|f|](x,1-t) - \big|\Heat[f](x,1-t)\big|\Big)w(x)\,dx \leq \eps \|f\|_{L_1(\tilde{w})}.
}
Let us assume the contrary: let there exist~$x_0\in B_R(0)$ such that
\eq{
\big|\Heat[f](x_0,1-t)\big| \leq \frac12 c_1\|f\|_{L_1(\tilde{w})}.
}
According to~\eqref{NoCanc}, 
\eq{
\Heat[|f|](x_0,1-t) - \big|\Heat[f](x_0,1-t)\big| > \frac{c_1}{2}\|f\|_{L_1(\tilde{w})}.
}
By Lemma~\ref{LipschitzLemma}, the expression on the left hand side of the above inequality is a Lipschitz function of~$x_0$ with the Lipschitz constant~$L\|f\|_{L_1(w)}$, where~$L$ depends on~$R$ and the parameters of~$\tilde{w}$ (those are~$\theta_1$ and~$c_{w}$; we have used Lemma~\ref{HeatingWeights} here again) only. Therefore, 
\eq{
\Heat[|f|](x,1-t) - \big|\Heat[f](x,1-t)\big| > \frac{c_1}{5}\|f\|_{L_1(\tilde{w})}
}
for any~$x\in B_R(0) \cap B_{\frac{c_1}{10 L}}(x_0)$. We integrate this inequality with respect to~$x$:
\eq{
\int\limits_{B_R(0) \cap R_{\frac{c_1}{10 L}}(x_0)}\!\!\!\!\!\!\!\! \Big(\Heat[|f|](x,1-t) - \big|\Heat[f](x,1-t)\big|\Big)w(x)\,dx \geq \frac{\pi_d}{d!}\Big(\frac{c_1}{10L}\Big)^d \big(\s[w](R)\big)^{-1}\frac{c_1}{5} \|f\|_{L_1(\tilde{w})},
}
here~$\pi_d$ is the volume of the unit ball in~$\R^d$ (we have assumed that~$R$ is larger than~$\frac{c_1}{10L}$, which is not a restriction, to estimate the volume of the intersection of two balls, the center of the smaller one lying inside the larger one, by~$(d!)^{-1}$ of the volume of the smaller one). The latter inequality contradicts~\eqref{UsefulFormulaRestated} provided
\eq{
\eps < \frac{\pi_d}{d!}\Big(\frac{c_1}{10L}\Big)^d \big(\s[w](R)\big)^{-1}\frac{c_1}{5}.
}
\end{proof}
\begin{Le}\label{FromConcentratedToTails}
Assume that
\eq{\label{eq524}
\|g\|_{L_1(u)} \leq C\|g\|_{L_1(\tilde{w})}
}
for some function~$g \in L_1(\tilde{w})$ and~$C  > 0$. Then\textup, there exists~$R > 0$ that depends on~$\gamma \in (0,1)$\textup,~$\theta_1$\textup,~$\theta_2$\textup, and~$C$ only such that 
\eq{\label{SmallTailGen}
\int\limits_{|x|\geq R}|g(x)|\tilde{w}(x)\,dx \leq \gamma \|g\|_{L_1(\tilde{w})}.
}
The choice of~$R$ is independent of~$t < \frac12$ \textup(this parameter is implicitly present in the definition~\eqref{Tildew} of~$\tilde{w}$\textup) and~$g$.
\end{Le}
In particular, if~$\gamma =\frac12$, the concentration condition~\eqref{Concentrated} implies~\eqref{SmallTail} with~$f:= f_2$ and~$t:=A^{-4}$.
\begin{proof}
We estimate the left hand side of~\eqref{SmallTailGen}:
\mlt{
\int\limits_{|x| \geq R} |g(x)|\tilde{w}(x)\,dx \leq \int\limits_{|x|\geq R}|g(x)|u(x)\,dx \sup\limits_{|x|\geq R}\frac{\tilde{w}(x)}{u(x)} \Lref{\hbox{\tiny Lem.~\ref{HeatingWeights}; \eqref{eq524}}}  \\C\|g\|_{L_1(\tilde{w})}\sup\limits_{|x|\geq R}\frac{C_{\tilde{w}} (1+|x|)^{-\theta_1}}{(1+|x|)^{-\theta_2}} \Lref{\theta_2 < \theta_1} CC_{\tilde{w}} (1+R)^{\theta_2 - \theta_1}\|g\|_{L_1(\tilde{w})},
}
and see that the quantity in front of~$\|g\|_{L_1(\tilde{w})}$ becomes arbitrarily small when~$R\to \infty$ since~$\theta_2 < \theta_1$. 
\end{proof}
\begin{Le}\label{Prokhorov}
Let~$\{L_R\}_{R \in \N}$ be a sequence of positive scalars. The set
\eq{
\LIP = \Set{g \in L_1(\tilde{w})}{\|g\|_{L_1(u)} \leq C\|g\|_{L_1(\tilde{w})} \leq C; \quad \forall R\in \N\quad \|g\|_{\Lip(B_R(0))} \leq L_R}
} 
is compact in~$L_1(\tilde{w})$.
\end{Le}
\begin{proof}
It suffices to construct a finite~$\epsilon$-net in~$L_1(\tilde{w})$ for~$\LIP$ given any~$\epsilon \in (0,1)$. By Lemma~\ref{FromConcentratedToTails}, there exists~$R_{\epsilon} > 0$ such that
\eq{
\int\limits_{|x| > R_\epsilon} |g(x)|\tilde{w}(x)\,dx \leq \frac{\epsilon}{2}\|g\|_{L_1(\tilde{w})} \leq \frac{\epsilon}{2}
}
for any~$g \in \LIP$. It is also clear that for any~$R\in\N$ we have~$\|g\|_{L_\infty(B_R(0))} \leq b_R$ for any~$g\in \LIP$, where~$\{b_R\}_R$ is a fixed sequence of constants. Thus, by the Arzela--Ascoli theorem, there is a finite~$\frac{\epsilon}{2\int\tilde{w}}$-net~$\{\tilde{h}_k\}_k$ for the set~$\LIP_{R_{\epsilon}}$ in the space~$C(B_{R_\epsilon}(0))$, where
\eq{
\LIP_R = \set{h\in C(B_R(0))}{\exists g \in \LIP\quad h = g|_{B_R(0)}}.
}
Then, the set~$\{h_k\}_k$, where
\eq{
h_k(x) = \begin{cases} \tilde  h_k(x),\quad &x\in B_{R_\epsilon}(0);\\
0,\quad &x\notin B_{R_{\epsilon}}(0),
\end{cases}
}
is an~$\epsilon$-net for~$\LIP$ in~$L_1(\tilde{w})$.
\end{proof}
\begin{proof}[Proof of Theorem~\ref{Compactness}.]
We use Lemma~\ref{FromConcentratedToTails} with~$g:= f_2$ and~$t := A^{-4}$ to choose~$R > \sqrt{d}$ such that~\eqref{SmallTail} holds true with~$f:=f_2$, i.e. such that
\eq{
\int\limits_{|x| \geq R} |f_2(x)|\Heat[w](\fdot,1-A^{-4})(x)\,dx \leq \frac12 \|f_2\|_{L_1(\Heat[w](\fdot,1-A^{-4}))}.
} 
By Lemma~\ref{AntiCancellationLemma} with~$f:= f_2$ and~$t := A^{-4}$ (the application of this lemma is legal since condition~\eqref{Milderflatness} in this case follows from the assumption that~$(0,0)$ is flat by virtue of Remark~\ref{AroundFlatness}),
\eq{
|f_0(x)| = \Big|\Heat[f_2](x,1-A^{-4})\Big| \geq c\|f_{2}\|_{L_1(\Heat[w](\fdot,1-A^{-4}))}, \quad x\in B_R(0),
}
provided~$\eps$ is sufficiently small. Note that the constant~$c$ depends neither on~$A$, nor on~$\eps$ (provided~$\eps$ is sufficiently small). This inequality justifies~\eqref{CubeFromBelow}.

Now we fix~$A$ and allow~$\eps$ to depend on~$A$. Our aim is to prove~\eqref{ImprovedMonotonicityNonPositive}. Assume the contrary: let there exist a sequence of functions~$f^n \in L_1\cap\WW$ such that the atom~$(0,0)$ is~$\frac{1}{n}$-flat for~$f^n$, the condition~\eqref{Concentrated} is fulfilled with~$f_2:=f_2^n$, however,~\eqref{ImprovedMonotonicityNonPositive} is violated in the sense that
\eq{\label{Contrary}
\|f_1^n\|_{L_p(\Heat[v](\fdot,\frac{1-A^{-2}}{p}))} > A^{\frac{d(p-1)}{p} - \frac14\tilde \delta} \|f_0^n\|_{L_p(v)}.
}
Without loss of generality, we assume
\eq{\label{UnitIntegral}
\|f_0^n\|_{L_1(w)} = 1.
}
Since~$(0,0)$ is~$\frac{1}{n}$-flat for~$f^n$,
\eq{\label{eq534}
\|f_3^n\|_{L_1(\Heat[w](\fdot,1-A^{-6}))} \leq 2,
}
and, by Remark~\ref{LipschitzRemark}, we also have
\eq{
\|f_2^n\|_{\Lip(B_R(0))} \leq L_R
}
for some fixed constants~$L_R$ (these constants do not depend on~$n$; they will surely depend on~$A$ since
\eq{
f_2^n = \Heat[f_3^n](\fdot,A^{-4} - A^{-6})\quad \hbox{and}\ s = A^{-4} - A^{-6}
}
in the terminology of Remark~\ref{LipschitzRemark}). By~\eqref{eq534},
\eq{
\|f_2^n\|_{L_1(\Heat[w](\fdot,1-A^{-4}))} \leq 2.
}
We apply Lemma~\ref{Prokhorov} and extract from the sequence~$\{f_2^n\}_{n}$ a subsequence that converges to a function~$F$ in~$L_1(\Heat[w](\fdot,1-A^{-4}))$. Without loss of generality, we may assume that~$\{f_2^n\}_n$ converges to~$F$ itself. Since the topology of~$L_1(\tilde{w})$ is stronger than that of~$\Sw'(\R^d,\R^\ell)$ (we use that~$\tilde{w}$ satisfies the bound from below as in~\eqref{BoundOnW}, thanks to Lemma~\ref{HeatingWeights}), we obtain~$F\in \WW$.

By Lemma~\ref{L1Lp} with~$u := \tilde{w}$ and~$v := w$, we have~$f_0^{n} \to \Heat[F](\fdot,1-A^{-4})$ in~$L_1(w)$. In particular, our normalization~\eqref{UnitIntegral} then implies~$F\ne 0$. Therefore, the flatness assumption leads to
\eq{
\|F\|_{L_1(\Heat[w](\fdot,1-A^{-4}))} = \|\Heat[F](\fdot,1-A^{-4})\|_{L_1(w)}.
}
By Lemma~\ref{JensenEquality},~$F = a\otimes h$, where~$a\in \R^\ell$ and~$h \geq 0$. Note that~$h\in \MM^\WW$.

On the other hand, 
\alg{
f_0^{n} \to \Heat[F](\fdot,1-A^{-4})\quad &\hbox{in }\ L_p(v);\\
f_1^n \to \Heat[F](\fdot, A^{-2}-A^{-4})\quad &\hbox{in}\ L_p\Big(\Heat[v]\Big(\fdot,\frac{1-A^{-2}}{p}\Big)\Big),
}
by Lemma~\ref{L1Lp} since we have assumed~$\theta_3 \geq p\theta_1$ (as usual, we have used Lemma~\ref{HeatingWeights} several times here). Thus,~\eqref{Contrary} implies
\eq{
\|\Heat[h](\fdot, A^{-2}-A^{-4})\|_{L_p(\Heat[v](\fdot,\frac{1-A^{-2}}{p}))} \geq A^{\frac{d(p-1)}{p} - \frac14\tilde \delta} \|\Heat[h](\fdot,1-A^{-4})\|_{L_p(v)},
}
which contradicts Corollary~\ref{RobustCorollary} since~$h \in \MM^\WW$ (we apply the corollary to the function~$H(x,\theta) := \Heat[h](x,\theta - A^{-4})$, the weight~$G:= v$ and~$t:= A^{-2}$).
\end{proof}
\begin{Cor}\label{DirectInductionStep}
By Lemmas~\textup{\ref{HeatingWeights}} and~\textup{\ref{L1Lp}},
\eq{
\|f_0\|_{L_p(v)} \lesssim_A \|f_2\|_{L_1(\Heat[w_{0,0}] (\fdot, 1-A^{-4}))},
}
thus\textup, if all the assumptions of Theorem~\textup{\ref{Compactness}} are satisfied\textup, then one may combine inequalities~\eqref{ImprovedMonotonicityNonPositive} and~\eqref{CubeFromBelow} into
\eq{\label{f537}
 \|f_1\|_{L_p(\Heat[v](\fdot,\frac{1-A^{-2}}{p}))} \lesssim_A A^{\frac{d(p-1)}{p} - \frac14\tilde \delta}\|f_0\|_{L_p(Q_{0,0})}.
}
Though the constant in this inequality does not depend on~$\eps$\textup, the inequality becomes valid only if~$\eps$ is sufficiently small\textup, and the required smallness of~$\eps$ may depend on~$A$. By Lemma~\textup{\ref{HeatingWeights},}~\eqref{f537} also implies
\eq{\label{CubeCube}
\|f_1\|_{L_p(Q_{0,0})} \lesssim_A A^{\frac{d(p-1)}{p} - \frac14\tilde \delta}\|f_0\|_{L_p(Q_{0,0})}.
}
 \end{Cor}

\section{Horizontal interaction}\label{S6}
Let~$(k,j)$ be an atom. It is convenient to introduce the notation
\eq{
f_{k,j}^* = \|f_{k+2}\|_{L_1(\Heat[w_{k,j}](\fdot, A^{-2k} - A^{-2k-4}))}.
}
The quantity~$f^*_{k,j}$ may be informally thought of as the size of the function~$f_k$ on the cube~$Q_{k,j}$ (or the size of the martingale~$f$ on the atom~$(k,j)$ at the moment~$k$). Note that (e.g. by Lemma~\ref{L1Lp} and the assumption~$f\in L_1$) the sequence~$\{f_{k,j}^*\}_{j\in\Z^d}$ is bounded for any~$k$, which makes the maximal functions introduced in the following definition finite. 
\begin{Def}
Let~$\theta_4 > d$ be a number we will specify later. Consider the collection of maximal functions~$\M_k^{\theta_4}\colon \Z^d\to \mathbb{R}^+$\textup,~$k \in\mathbb{N}\cup\{0\}$\textup, defined as follows\textup:
\eq{
\M^{\theta_4}_{k,j}[f] = \sup\limits_{i\in\Z^d}(1+|i-j|)^{-\theta_4}f_{k,i}^*,\quad j \in \Z^d.
}
\end{Def}
Similar smoothing maximal function were used in~\cite{BourgainBrezis2007}. This particular definition is borrowed from~\cite{TaoShortStoryUchiyama2007}. 

We fix~$k,\theta_4$, and~$f$ for a while. We suppress these parameters in our notation if this does not lead to ambiguity and simply write~$f^*_j$ and~$\M_j$. The maximal operator we have introduced generates an interesting oriented graph.
\begin{Def}
By the horizontal graph~$\tilde\Gamma_k$ we mean the following oriented graph. Fix some number~$\lambda > 1$ that will be specified slightly later \textup(we will require~$\lambda$ to be close to one\textup). The set of vertices of~$\tilde\Gamma_k$ is the lattice~$\Z^d$. For each point~$j$\textup, we find some point~$\vec{j} \in \Z^d$ such that
\eq{
\M_{k,j}[f] \leq \lambda (1+|\vec{j} - j|)^{-\theta_4}f^*_{k,\vec j}.
}
If~$\vec j = j$\textup, we do nothing. In the other case\textup, we draw an arrow from~$\vec j$ to $j$.
\end{Def}
Note that a vertex has at most one incoming arrow. Informally, the arrow~$\vec j \to j$ signifies that the point~$\vec j$ dominates~$j$ in the sense that we may estimate the local size~$f_j^*$ by~$f_{\vec j}^*$ uniformly.

\begin{Le}\label{GoodGraph}
If~$\lambda$ is sufficiently close to one \textup(depending on~$\theta_4$ only and independent of~$f$\textup)\textup, then there are no oriented paths of length~$2$ in~$\tilde\Gamma_k$.
\end{Le}
\begin{proof}
Assume the contrary, let there be a path of length two. Without loss of generality, we may assume that the path is~$j\to 0 \to i$, where~$i\ne 0$ and~$j\ne 0$ by construction. Then,
\eq{
\M_i \leq \lambda (1+|i|)^{-\theta_4}f_0^*\quad \hbox{and}\quad \M_0\leq \lambda (1+|j|)^{-\theta_4}f_j^*,
}
which leads to the inequality
\eq{
\M_i \leq \lambda^2 \big((1+|i|)(1+|j|)\big)^{-\theta_4} f_j^*.
}
Combining this with the definition of~$\M_i$ and noting that~$f_j^* > 0$, we arrive at
\eq{
(1+|i-j|)^{-\theta_4} \leq \lambda^2\big((1+|i|)(1+|j|)\big)^{-\theta_4},
}
which is equivalent to
\eq{
1+|i-j| \geq \lambda^{-\frac{2}{\theta_4}} (1+|i| + |j| + |i||j|).
}
We note that~$|i||j| \geq \frac13(1+|i|+|j|)$ for any~$i,j \in \Z^d\setminus \{0\}$, so,
\eq{
1+|i-j| \geq \frac43\lambda^{-\frac{2}{\theta_4}}(1+|i| + |j|),
}
which is definitely false if~$\lambda < \big(\frac43\big)^{\frac{\theta_4}{2}}$.
\end{proof}
In particular, there are no pairs of arrows~$i\to j$ and~$j \to i$ in~$\tilde\Gamma_k$, so it is indeed an oriented graph. We fix
\eq{\label{Lambda}
\lambda = \min\Big(2, \frac{1+ \big(\frac43\big)^{\frac{\theta_4}{2}}}{2}\Big).
}
\begin{Def}
Let~$K > 1$ be a real. We say that an atom~$(k,j)$ is~$K$-saturated if
\eq{
\M_{k,j}^{\theta_4}[f] \leq K f_{k,j}^*.
}
\end{Def}
\begin{Le}\label{NoArrowsSaturated}
If the vertex~$j$ does not have an incoming arrow in~$\tilde\Gamma_k$\textup, then the atom~$(k,j)$ is~$2$-saturated. If~$(k,j)$ is not~$2$-saturated\textup, then it has an incoming arrow.
\end{Le}
\begin{proof}
Let us prove the first claim. By construction, if~$(k,j)$ does not have an incoming arrow, then
\eq{
\M_j \leq \lambda f_j^* < 2f_j^*,
} 
which means~$(k,j)$ is~$2$-saturated. The second claim is similar: if~$(k,j)$ is not~$2$-saturated, then
\eq{
2f_j^* < \M_j,
}
and by construction,~$(k,j)$ must have an incoming arrow.
\end{proof}
Lemmas~\ref{GoodGraph} and~\ref{NoArrowsSaturated} imply the following statement.
\begin{Cor}\label{SaturatedArrowCor}
If a vertex in~$\tilde{\Gamma}_k$ has an outgoing arrow\textup, then the corresponding atom is~$2$-saturated.
\end{Cor}
Recall the weight~$u$ defined in~\eqref{uformula}.
\begin{Le}\label{SaturatedConcentrated}
Let~$\theta_2 > \theta_4 + d$. There exists a constant~$C$ depending on~$\theta_1, \theta_2$\textup, and~$\theta_4$ only such that if~$(0,0)$ is~$2$-saturated\textup, then it fulfills the concentration assumption
\eq{
\|f_{2}\|_{L_1(u)} \leq Cf^*_{0,0}.
}
\end{Le}
\begin{proof}
It suffices to prove
\eq{
\|f_{2}\|_{L_1(u)} \leq \frac{C}{2}\M_{0,0}^{\theta_4}[f].
}
It remains to write several inequalities:
\mlt{
\int\limits_{\R^d}|f_2(x)|u(x)\,dx = \sum\limits_{i\in\Z^d}\int\limits_{Q_{0,i}} |f_2(x)|u(x)\,dx \Leqref{uformula}\\ \s[u](\sqrt{d}) \sum\limits_{i\in\Z^d} (1+|i|)^{-\theta_2}\int\limits_{Q_{0,i}}|f_2(x)|\,dx \Lsref{\hbox{\tiny Lem.~\ref{SmoothnessConv}}}
\frac{\s[u](\sqrt{d})\s[w](\sqrt{d})}{\tilde{w}(0)} \sum\limits_{i\in\Z^d} (1+|i|)^{-\theta_2}f_{0,i}^* \leq\\ \frac{\s[u](\sqrt{d})\s[w](\sqrt{d})}{\tilde{w}(0)} \sum\limits_{i\in\Z^d} (1+|i|)^{\theta_4-\theta_2}\M_{0,0}^{\theta_4}[f] \lesssim \M_{0,0}^{\theta_4}[f];
}
as usual, we are using the notation~$\tilde{w} = \Heat[w](\fdot,1 - A^{-4})$.
\end{proof}
\begin{Le}\label{SubordinationL1Lp}
Let the atom~$(k,i)$ be subordinate to~$(k,j)$ in the sense that there is the arrow~$j \to i$ in the graph~$\tilde\Gamma_k$. Then\textup, the inequality
\eq{
\|f_{k+1}\|_{L_p(Q_{k,i})} \lesssim_A A^{\alpha (k+1)}(1+|i-j|)^{-\theta_4}f^*_{k,j}
}
holds true\textup; the constant in the inequality depends neither on~$f$\textup, nor on~$A$\textup, nor on the particular choice of~$i$ and~$j$.
\end{Le}
\begin{proof}
Without loss of generality, we may assume~$k=0$. We apply Corollary~\ref{EarlyCor}:
\eq{
\|f_1\|_{L_p(Q_{0,i})} \lesssim_A A^{\alpha}\|f_2\|_{L_1(\Heat[w_{0,i}](\fdot,1-A^{-4}))} = A^{\alpha}f^*_{0,i} \leq A^{\alpha}\M_{0,i}^{\theta_4}[f] \Lref{j\to i} \lambda A^{\alpha} (1+|i-j|)^{-\theta_4} f_{0,j}^*.
}
\end{proof}
\begin{Th}\label{InductionStep}
Let~$\delta_0\notin \MM^\WW$\textup, let the parameters~$\theta_1, \theta_2, \theta_3, \theta_4$, and~$p$ be fixed. Let these parameters satisfy
\alg{
p \leq 2;\\
\theta_1 > \theta_2;\\
\theta_3 \geq p\theta_1;\\
\theta_2 > \theta_4 + d;\\
\theta_4 > d.
}
Let also~$\theta_5$ be a fixed parameter such that~$d < \theta_5 < \theta_4$. The following statement is true for any~$A$ sufficiently large. There exists~$\eps > 0$\textup, possibly depending on~$A$\textup, and a positive constant~$\delta^*$ independent of~$A$ such that if the atom~$(k,j)$ is~$\eps$-flat and~$2$-saturated whereas the atom~$(k,i)$ is subordinate to~$(k,j)$ in the graph~$\tilde\Gamma_k$\textup, then 
\eq{
\|f_{k+1}\|_{L_p(Q_{k,i})} \lesssim_A A^{\alpha - \delta^*}(1+|i-j|)^{-\theta_5}\|f_{k}\|_{L_p(Q_{k,j})}.
}
\end{Th}
\begin{proof}
Without loss of generality, we may assume~$k=0$ and~$j=0$. The desired inequality will follow from the two inequalities below (recall~$\tilde{\delta}$ from Corollary~\ref{RobustCorollary}):
\alg{
\label{First}\|f_1\|_{L_p(Q_{0,i})} \lesssim_A A^{\alpha}(1+|i|)^{-\theta_4}\|f_0\|_{L_p(Q_{0,0})};\\
\label{Second}\|f_1\|_{L_p(Q_{0,i})} \lesssim_A A^{\alpha - \frac{\tilde \delta}{4}}(1+|i|)^{\frac{\theta_3}{p}}\|f_0\|_{L_p(Q_{0,0})}. 
}
According to Lemma~\ref{SaturatedConcentrated}, the atom~$(0,0)$ fulfills the concentration condition~\eqref{Concentrated}. Thus, the application of Theorem~\ref{Compactness} and Corollary~\ref{DirectInductionStep} to the atom~$(0,0)$ is legal.

Inequality~\eqref{First} is a consequence of Lemma~\ref{SubordinationL1Lp} and~\eqref{CubeFromBelow}.

Let us prove~\eqref{Second}:
\mlt{
\|f_1\|_{L_{(Q_{0,i})}}^p \Lsref{\hbox{\tiny Lem.~\ref{HeatingWeights}}}\!\!\!\!\!\!_A \s[v](\sqrt{d})(1+|i|)^{\theta_3}\int\limits_{Q_{0,i}}|f_1(x)|^p\Heat[v]\Big(x,\frac{1 - A^{-2}}{p}\Big)dx \leq \\
\s[v](\sqrt{d}) (1+|i|)^{\theta_3} \|f_1\|_{L_p(\Heat[v](x,\frac{1 - A^{-2}}{p}))}^p \Lseqref{f537}\!\!\!\!\!_A A^{p(\alpha - \frac{\tilde \delta}{4})} (1+|i|)^{\theta_3}\|f_0\|_{L_p(Q_{0,0})}^p.
}

\end{proof}

\section{Vertical interaction and control of flat atoms}\label{S7}
Now we will introduce a graph~$\Gamma$ that will express the vertical domination of atoms. In this graph, the arrows will always go down, i.e. from an atom~$(k,j)$ to~$(k+1,j')$.
\begin{Def}\label{VerticalGraph}
The set of vertices of~$\Gamma$ is the set of all~$\eps$-flat atoms. We draw an arrow from~$(k,j)$ to~$(k+1,j')$ if~$(k,j)$ is~$2$-saturated and either~$Q_{k+1,j'} \subset Q_{k,j}$ or~$Q_{k+1,j'}\subset Q_{k,i}$\textup,~$Q_{k,i}$ is not~$2$-saturated\textup, and~$j\to i$ in~$\tilde\Gamma_k$.
\end{Def}
It will be important for our considerations that any two cubes~$Q_{k,j}$ and~$Q_{k',j'}$ are either disjoint up to a set of measure zero, or one of them contains the other (this follows from the assumption that~$A$ is odd).
\begin{Rem}
Note that~$\Gamma$ is a forest in the sense that it is a disjoint union of maximal by inclusion oriented trees~$\T_1,\T_2,\ldots$. Indeed\textup, there are no non-oriented cycles in~$\Gamma$ since every vertex has at most one incoming arrow and all the arrows \textup"go down\textup" \textup(i.e. from a vertex~$(k,j)$ to~$(k+1,j')$\textup). By~$(k_q,j_q)$ we denote the atom corresponding to the root of~$\T_q$\textup, $q\in \N$. 
\end{Rem}
\begin{Rem}\label{OnlySaturatedOutGoing}
By Definition~\textup{\ref{VerticalGraph}} and Corollary~\textup{\ref{SaturatedArrowCor}}\textup, only~$2$-saturated atoms might have outgoing arrows in~$\Gamma$. Thus\textup, inside a tree~$\T_q$ only the leaves might be not $2$-saturated.
\end{Rem}
In Lemma~\ref{Lele} and Corollary~\ref{LeleCor}, we assume that~$\eps$ is sufficiently small (depending on~$A$).
\begin{Le}\label{Lele}
Assume~$\delta_0\notin \MM^\WW$. Let~$(k,j)$ be a flat atom\textup, let~$\{(k+1,i)\}_{i\in J}$ be all its kids in~$\Gamma$. Then\textup,
\eq{
\|f_{k+1}\|_{L_p(\bigcup\limits_{i\in J} Q_{k+1,i})} \lesssim_A A^{\alpha - \delta^*}\|f_k\|_{L_p(Q_{k,j})}.
} 
\end{Le}
\begin{proof}
Without loss of generality, we may assume~$k=0$ and~$j=0$. By construction,
\eq{
\bigcup\limits_{i\in J} Q_{1,i} \subset \Omega_{0,0} := Q_{0,0} \cup \Big(\!\!\!\!\!\!\bigcup\limits_{j\colon\scriptscriptstyle(0,0)\stackrel{\tilde{\Gamma}_0}{\to} (0,j)}\!\!\!\!\!\!Q_{0,j}\Big).
}
By Remark~\ref{OnlySaturatedOutGoing},~$(0,0)$ is~$2$-saturated (otherwise it has no kids in~$\Gamma$ and there is nothing to prove), so it is legal to apply Theorems~\ref{Compactness} and~\ref{InductionStep}. Then,
\mlt{
\|f_1\|^p_{L_p(\bigcup\limits_{i\in J} Q_{1,i})} \leq \|f_1\|^p_{L_p(\Omega_{0,0})}\Lsref{\hbox{\tiny \eqref{CubeCube}; Th~\ref{InductionStep}}}\!\!\!\!\!\!\!\!\!\!\!\!_A\\
A^{p(\alpha - \delta^*)} \Big(1+ \sum\limits_{j\in \Z^d} (1+|j|)^{-p\theta_5}\Big) \|f_0\|^p_{L_p(Q_{0,0})} \lesssim_A A^{p(\alpha - \delta^*)}\|f_0\|^p_{L_p(Q_{0,0})}.
}
\end{proof}
Let~$\delta^{**} = \frac12 \delta^*$. The next two corollaries are immediate consequences if Lemma~\ref{Lele}, they do not need proofs. 
\begin{Cor}\label{LeleCor}
Assume~$\delta_0\notin \MM^\WW$. Let~$(k,j)$ be a flat atom\textup, let~$\{(k+1,i)\}_{i\in J}$ be all its kids in~$\Gamma$. Then\textup,
\eq{\label{Submult}
\|f_{k+1}\|_{L_p(\bigcup\limits_{i\in J} Q_{k+1,i})} \leq  A^{\alpha - \delta^{**}}\|f_k\|_{L_p(Q_{k,j})},
} 
provided~$A$ is sufficiently large.
\end{Cor}
Now we fix~$A$ so large that~\eqref{Submult} holds true. We fix~$\eps$ to be as small as prescribed by Theorems~\ref{Compactness} and~\ref{InductionStep} for our particular choice of~$A$.
\begin{Cor}
Assume~$\delta_0\notin \MM^\WW$. Let~$\T_q$ be a maximal by inclusion tree in~$\Gamma$\textup, let~$(k_q,j_q)$ be its root. Then\textup, for any~$N\in\N$\textup,
\eq{\label{Inductive}
\|f_{k_q+N}\|_{L_p(\!\!\!\!\!\!\!\!\!\bigcup\limits_{(k_q+N,i)\in\T_q}\!\!\!\!\!\!\!\!\! Q_{k_q+N,i})} \leq A^{(\alpha - \delta^{**})N} \|f_{k_q}\|_{L_p(Q_{k_q,j_q})}.
}
\end{Cor}
Let us investigate the roots of our trees. 
\begin{Le}\label{root}
Let~$(k_q,j_q)$ be the root of~$\T_q$ and let~$k_q \geq 1$. Then\textup, the atom~$(k_q-1,j')$ such that~$Q_{k_q,j_q}\subset Q_{(k_q-1,j')}$ is either~$\eps$-convex or is subordinate to an~$\eps$-convex atom in~$\tilde\Gamma_{k_q-1}$.
\end{Le}
\begin{proof}
Assume the atom~$(k_q-1,j')$ is~$\eps$-flat, otherwise there is nothing to prove. To prove the lemma, we need to show~$(k_q-1,j')$ is subordinate to a convex atom in~$\tilde\Gamma_{k_q-1}$. Note that~$(k_q-1,j')$ is not~$2$-saturated, because there is no arrow~$(k_{q}-1,j')\to (k_q,j_q)$ in~$\Gamma$. Therefore, by Lemma~\ref{NoArrowsSaturated}, this atom is subordinated to another atom~$(k_q-1,\vec{j})$ in~$\tilde\Gamma_{k_q-1}$. Corollary~\ref{SaturatedArrowCor} says~$(k_q-1,\vec{j})$ is~$2$-saturated, and thus, it cannot be~$\eps$-flat since there is no arrow~$(k_q-1,\vec{j})\to (k_q,j_q)$ in~$\Gamma$. Therefore,~$(k_q-1,\vec{j})$ is convex, and the lemma is proved.
\end{proof}
We remind the reader that since we have already fixed~$A$, all the constants in our inequalities are now allowed to depend on~$A$. However, we still care about the uniformity with respect to~$N$.
\begin{St}\label{LebesgueInd}
Assume~$\delta_0\notin \MM^\WW$. Let~$\{\T_q\}_q$ be all the trees that start their development at the level~$K \geq 1$\textup, i.e.  such that~$k_q = K$. Then\textup, for any~$N \in \N$\textup,
\eq{\label{TreesSummation}
\sum\limits_q\|f_{K+N}\|_{L_p(\!\!\!\!\!\!\!\bigcup\limits_{(K+N,i) \in \T_q}\!\!\!\!\!\!\! Q_{K+N,i})} \lesssim_N A^{(\alpha - \delta^{**})N + \alpha K} \Big(\|f_{K+2}\|_{L_1} - \|f_{K-1}\|_{L_1}\Big),
}
where the constant in this inequality is independent of~$N$.
\end{St}
\begin{proof}
We rely upon Lemma~\ref{root} and analyze the two cases arising in this lemma separately. Let~$\T_q$ be some tree with the root~$(K,j_q)$. Let~$Q_{K,j_q} \subset Q_{K-1,j'}$.

Consider the first case:~$(K-1,j')$ is convex. In this case,
\mlt{\label{FirstTree}
\|f_{K+N}\|_{L_p(\!\!\!\!\!\!\!\bigcup\limits_{(K+N,i) \in \T_q}\!\!\!\!\!\!\! Q_{K+N,i})}\Leqref{Inductive} A^{(\alpha - \delta^{**})N}\|f_{K}\|_{L_p(Q_{K,j_q})} \Lsref{\hbox{\tiny Cor.~\ref{L1LpFormulaRescaled}}}\\ A^{(\alpha - \delta^{**})N+ \alpha K}\|f_{K+2}\|_{L_1(\Heat[w_{K -1, j'}](\fdot,A^{-2K+2} - A^{-2K - 4}))} \Lsref{\scriptscriptstyle (K-1,j')\in \Conv}\\A^{(\alpha - \delta^{**})N +\alpha K} \Big(\|f_{K+2}\|_{L_1(\Heat[w_{K -1, j'}](\fdot,A^{-2K+2} - A^{-2K - 4}))} - \|f_{K-1}\|_{L_1(w_{K-1,j'})}\Big).
}

Consider the second case: let now~$(K-1,j')$ be subordinated to a convex atom~$(K-1,\vec{j})$ in~$\tilde\Gamma_{K-1}$. In this case,
\mlt{\label{SecondTree}
\|f_{K+N}\|_{L_p(\!\!\!\!\!\!\!\bigcup\limits_{(K+N,i) \in \T_q}\!\!\!\!\!\!\!Q_{K+N,i})}\Leqref{Inductive} A^{(\alpha - \delta^{**})N}\|f_{K}\|_{L_p(Q_{K,j_q})} \Lsref{\hbox{\tiny Cor.~\ref{L1LpFormulaRescaled}}}\\ A^{(\alpha - \delta^{**})N+\alpha K}\|f_{K+1}\|_{L_1(\Heat[w_{K -1, j'}](\fdot,A^{-2K+2} - A^{-2K - 2}))} =\\ A^{(\alpha - \delta^{**})N+\alpha K} f_{K-1,j'}^* \!\!\!\!\!\!\!\!\Lsref{\scriptscriptstyle (K-1,\vec{j})\stackrel{\scriptscriptstyle\tilde \Gamma_{K-1}}{\xrightarrow{\hspace*{0.5cm}}} (K-1,j')} A^{(\alpha - \delta^{**})N+\alpha K} (1+|j'-\vec{j}|)^{-\theta_4} f_{K-1,\vec{j}}^*=\\ A^{(\alpha - \delta^{**})N+\alpha K}(1+|j'-\vec{j}|)^{-\theta_4}\|f_{K+1}\|_{L_1(\Heat[w_{K -1, \vec{j}}](\fdot,A^{-2K+2} - A^{-2K - 2}))}\!\!\!\!\!\Lsref{\hbox{\tiny Lem.~\ref{FirstMonotonicityLem}}}\\
A^{(\alpha - \delta^{**})N+\alpha K}(1+|j'-\vec{j}|)^{-\theta_4}\|f_{K+2}\|_{L_1(\Heat[w_{K -1, \vec{j}}](\fdot,A^{-2K+2} - A^{-2K - 4}))}\!\!\!\!\! \Lsref{\scriptscriptstyle (K-1,\vec{j})\in \Conv}\\
A^{(\alpha - \delta^{**})N+\alpha K}(1+|j'-\vec{j}|)^{-\theta_4}\Big(\|f_{K+2}\|_{L_1(\Heat[w_{K -1, \vec{j}}](\fdot,A^{-2K+2} - A^{-2K - 4}))} - \|f_{K-1}\|_{L_1(w_{K-1,\vec{j}})}\Big).
}
We sum the estimates~\eqref{FirstTree} and~\eqref{SecondTree} over all trees~$\T_q$ that have roots on the level~$K$. On the left hand side, we obtain the quantity we want to estimate in~\eqref{TreesSummation}. The quantity on the right is bounded by
\eq{\label{Terrible}
A^{d+(\alpha - \delta^{**})N+\alpha K}\sum\limits_{\vec{j}\in\Z^d}\sum\limits_{j'\in\Z^d} (1+|\vec{j}-j'|)^{-\theta_4} \Big(\|f_{K+2}\|_{L_1(\Heat[w_{K -1, \vec{j}}](\fdot,A^{-2K+2} - A^{-2K - 4}))} - \|f_{K-1}\|_{L_1(w_{K-1,\vec{j}})}\Big)
}
since any cube~$Q_{K-1,j'}$ contains at most~$A^d$ cubes of the next generation. We recall that~$\theta_4 > d$, and thus, the sum with respect to~$j'$ is bounded by a constant. It remains to use that the weights~$w_{K-1,\vec{j}}$ form a partition of unity to bound~\eqref{Terrible} by the right hand side of~\eqref{TreesSummation}.
\end{proof}
\begin{Rem}\label{LebRem}
In the case~$K = 0$\textup, the inequality~\eqref{TreesSummation} is replaced with
\eq{
\sum\limits_q\|f_{N}\|_{L_p(\!\!\!\!\bigcup\limits_{(N,i) \in \T_q} \!\!\!\!Q_{N,i})} \lesssim_N A^{(\alpha - \delta^{**})N} \|f_{2}\|_{L_1}.
}
\end{Rem}
\begin{St}\label{PreFinal}
Assume~$\delta_0\notin \MM^\WW$. Let~$\{\T_q\}_q$ be all the trees that start their development at the level~$K \geq 1$\textup, i.e.  such that~$k_q = K$. Then\textup, for any~$N > 0$\textup,
\eq{
\|f_{K+N}\|_{L_{p,1}(\!\!\!\!\!\!\!\!\bigcup\limits_{(K+N,i) \in \cup_q\T_q}\!\!\!\!\!\!\!\! Q_{K+N,i})} \lesssim_N A^{(\alpha - \delta^{***})N + \alpha K} \Big(\|f_{K+2}\|_{L_1} - \|f_{K-1}\|_{L_1}\Big),
}
where~$\delta^{***} > 0$ is a fixed real and the constant in the inequality does not depend on~$N$.
In the case~$K=0$\textup, we have
\eq{
\|f_{N}\|_{L_{p,1}(\!\!\!\!\!\!\!\!\bigcup\limits_{(N,i) \in \cup_q\T_q}\!\!\!\!\!\!\!\! Q_{N,i})} \lesssim_N A^{(\alpha - \delta^{***})N} \|f_{2}\|_{L_1}.
}
\end{St}
\begin{proof}
First, Lemma~\ref{Split} allows to derive similar estimates where the Lorentz norm is replaced with the~$L_p$ norm from Proposition~\ref{LebesgueInd} and Remark~\ref{LebRem} (with~$\delta^{***} = \delta^{**}$). Second, we note that all our combinatorial considerations (definitions of atoms and constructions of graphs) do not depend on~$p$ if we assume~$A$ to be sufficiently large. Therefore, the trivial interpolatory inequality
\eq{
\|g\|_{L_{p,1}}\lesssim \|g\|_{L_{p_1}}^\frac12 \|g\|_{L_{p_2}}^\frac12,
}
where~$p_1$ and~$p_2$ are small perturbations of~$p$ satisfying~$\frac{1}{p_1} + \frac{1}{p_2} = \frac{2}{p}$, allows to deduce the desired Lorentz bounds from the already obtained bounds on the~$L_{p_1}$ and~$L_{p_2}$ norms (we obtain~$\delta^{***}$ is the arithmetical mean of~$\delta^{**}$ for~$p_1$ and~$\delta^{**}$ for $p_2$).
\end{proof}
The next corollary follows immediately from Proposition~\ref{PreFinal}: one needs to compute the sum of a geometric series.
\begin{Cor}\label{Final}
Assume~$\delta_0\notin \MM^\WW$. Let~$\{\T_q\}_q$ be all the trees that start their development at the level~$K$. Then\textup, 
\eq{
\sum\limits_{N \geq 0} A^{-\alpha (K+N)}\|f_{K+N}\|_{L_{p,1}(\!\!\!\!\!\!\bigcup\limits_{(K+N,i) \in \cup_q\T_q}\!\!\!\!\!\! Q_{K+N,i})} \lesssim \Big(\|f_{K+2}\|_{L_1} - \|f_{K-1}\|_{L_1}\Big), \quad K \geq 1,
}
and
\eq{
\sum\limits_{N \geq 0} A^{-\alpha N}\|f_{N}\|_{L_{p,1}(\!\!\!\!\bigcup\limits_{(N,i) \in \cup_q\T_q}\!\!\!\! Q_{N,i})} \lesssim \|f_{2}\|_{L_1}.
}
\begin{proof}[Proof of Theorem~\ref{Main}.]
By Remark~\ref{plessthantwo}, it suffices to consider the case~$p \leq 2$. By Remark~\ref{ALesOne}, it suffices to prove inequality~\eqref{Besov2}. We choose the parameters
\alg{
\theta_1 &= 2d+4;\\
\theta_2 &= 2d+3;\\
\theta_3 &= 4d+9;\\
\theta_4 &= d+2;\\
\theta_5 &= d+1
}
and see that they fulfill all our previous requirements. This allows us to choose~$\lambda$ (see~\eqref{Lambda}) and~$A$ (this parameter is chosen to be sufficiently large in order~\eqref{Submult} to be true for~$p, p_1,$ and~$p_2$ in the proof of Proposition~\ref{PreFinal}). We also choose~$\eps$ as prescribed by Theorems~\ref{Compactness} (with~$C$ coming from Lemma~\ref{SaturatedConcentrated}) and~\ref{InductionStep}. This provides us with the sets~$\Conv$ and~$\Fl$ and the graphs~$\{\tilde\Gamma_k\}_k$ and~$\Gamma$. By Lemma~\ref{Split}, it suffices to prove the estimates
\alg{
\sum\limits_{k\geq 0}A^{-\alpha k}\|f_k\|_{L_{p,1}(\!\!\!\bigcup\limits_{(k,j)\in \Conv}\!\!\! Q_{k,j})}\lesssim \|f\|_{L_1};\\
\sum\limits_{k\geq 0}A^{-\alpha k}\|f_k\|_{L_{p,1}(\!\!\!\bigcup\limits_{(k,j)\in \Fl}\!\!\! Q_{k,j})}\lesssim \|f\|_{L_1}.
}
The first inequality is established in Proposition~\ref{ConvexControl}. The second one follows from Corollary~\ref{Final} and formula~\eqref{telescope} since any flat atom is a vertex in~$\Gamma$:
\mlt{
\sum\limits_{k\geq 0}A^{-\alpha k}\|f_k\|_{L_{p,1}(\!\!\!\bigcup\limits_{(k,j)\in \Fl}\!\!\! Q_{k,j})} \Lref{\hbox{\tiny Lem.~\ref{Split}}} \sum\limits_{k\geq 0} A^{-\alpha k}\sum\limits_{K \leq k} \|f_{k}\|_{L_{p,1}(\!\!\!\!\!\bigcup\limits_{(k,j)\in\!\!\!\! \bigcup\limits_{k_q = K}\!\!\!\! \T_q}\!\!\!\!\! Q_{k,j})} =\\ \sum\limits_{K \geq 0} \sum\limits_{N \geq 0} A^{-\alpha (K+N)}\|f_{K+N}\|_{L_{p,1}(\!\!\!\!\!\!\!\!\!\!\bigcup\limits_{(K+N,j)\in\!\!\!\! \bigcup\limits_{k_q = K}\!\!\!\! \T_q}\!\!\!\!\! Q_{K+N,j})} \lesssim \|f_{2}\|_{L_1} + \sum\limits_{K \geq 1}\Big(\|f_{K+2}\|_{L_1} - \|f_{K-1}\|_{L_1}\Big) \lesssim \|f\|_{L_1}.
} 
\end{proof}

\end{Cor}

\bibliography{mybib}{}
\bibliographystyle{amsplain}

St. Petersburg State University, Department of Mathematics and Computer Science;

St. Petersburg Department of Steklov Mathematical Institute;

d.m.stolyarov at spbu dot ru.
\end{document}